\documentclass[a4paper]{amsart}
\usepackage{amsmath,amsthm,amssymb,latexsym,epic,bbm,comment,mathrsfs}
\usepackage{graphicx,enumerate,stmaryrd, xcolor, color}
\usepackage{longtable}
\usepackage{diagbox}
\usepackage{booktabs}
\usepackage{mathrsfs}
\usepackage{amsmath}
\usepackage[all,2cell]{xy}
\usepackage{anysize}\marginsize{32mm}{32mm}{30mm}{30mm}

\xyoption{2cell}

\newtheorem{definition}{Definition}
\newtheorem{theorem}{Theorem}
\newtheorem{lemma}[theorem]{Lemma}

\newtheorem{proposition}[theorem]{Proposition}
\newtheorem{conjecture}[theorem]{Conjecture}
\newtheorem{example}[theorem]{Example}

\newtheorem{remark}[theorem]{Remark}

\usepackage[all]{xy}
\usepackage[active]{srcltx}
\usepackage[parfill]{parskip}
\usepackage{enumerate}

\usepackage{hyperref}
\newcommand{\mc}{\mathcal}
\newcommand{\mf}{\mathfrak}

\newcommand{\g}{\mathfrak{g}}

\newcommand{\h}{\mathfrak{h}}

\begin{document}
\title[Kronecker coefficients and Harrison centers of representation ring]{Kronecker coefficients and Harrison centers of\\ the representation ring of the symmetric group:\\a computational approach}

\author{Jiacheng Sun$^{2}$, Chi Zhang$^{1}$*, and Haoran Zhu$^{3}$}

\date{\today}

\begin{abstract}
We present a computational approach to studying the structure of the representation ring of the symmetric group in dimension six. The Kronecker coefficients and all power formulae of irreducible representations of $S_6$ are computed using the character theory of finite groups. In addition, considering direct sum decomposition of tensor products of different irreducible representations of $S_6$, we characterise generators of the representation ring $\mathcal{R}(S_6)$, show that its unit group $U(\mathcal{R}(S_6))$ is a Klein four-group, and related results on the structure of primitive idempotents. Furthermore, we introduce the Harrison centre theory to study the representation ring and show that the Harrison centre of the cubic form induced by the generating relations of $\mathcal{R}(S_6)$ is isomorphic to itself. Finally, we conclude with some open problems for future consideration.
\end{abstract}

\maketitle



\vspace{5mm}

\tableofcontents
    
\section{Introduction}
Kronecker coefficients represent a confluence of ideas from representation theory, algebraic combinatorics, and the burgeoning field of complexity theory. They count the multiplicities of irreducible representations in the tensor product of two other irreducible representations of the symmetric group. In particular, let \( M_{\lambda} \) denote the irreducible representation of \( S_n \) associated with the partition \( \lambda \). For partitions \( \lambda, \mu, \) and \( \nu \) of \( n \), \( g_{\lambda\mu\nu} \) denotes the multiplicity of \( M_{\nu} \) within the tensor product \( M_{\lambda} \otimes M_{\mu} \). The so-called {\em Kronecker problem}, a prominent challenge in algebraic combinatorics, attempts to determine a positive combinatorial formula for these coefficients. However, we often tend to treat special cases because their coefficients are difficult to calculate \cite{BI}.  (See \cite{Ball,Bla,Rem1,Rem2,Rosa} for some known special cases).

One possible idea is to look for rules by calculating some lower-order cases. The earlier part of our paper did this computational work as a preparation. To solve the Kronecker coefficients of the symmetric group $S_6$, we first focus on the specific properties of the character table and then perform a direct sum decomposition of the irreducible tensor product of $S_6$ based on the orthogonality of the characters. 
In the process, we observe a recurrence relation for the $n$-th power formulas of the irreducible representations of $S_6$, thus transforming the problem of solving the coefficients of their power formulas into solving the $(n-1)$-th power of a matrix A.  This makes the solution of the Kronecker coefficients simple and elementary, and it is sufficient to deal with the $n$-th power of the diagonal matrix after similarly diagonalising the original matrix.

In the process of computing the Kronecker coefficients, we must first obtain a decomposition of the tensor product of any two irreducible representations, which naturally inspired us to consider the structure of the representation ring of $S_6$. The concept of representation rings was first introduced by Green \cite{Green} in 1962, hence it is also called {\em Green's ring}. 
There are several results on different types of representation rings, for example classical groups, such as those of unitary groups, were introduced in \cite{Hu}; Minami \cite{Mi} explicitly gave the representation rings of orthogonal groups, later Chen \cite{chen1} studied the Green's ring of the Drinfel'd quantum double of Sweedler's four-dimensional Hopf algebras, they \cite{chen2} also computed the case of Taft algebras; the representation ring of the special linear group $SL(2,\mathbb{C})$ was described by Hamraoui and Khadija \cite{Ham}. For the symmetric group $S_n$, for example, Marian pointed out that for any symmetric group its representation ring is generated by the outer product of its hook-type irreducible representation, i.e. the $(n-1)$-dimensional natural representation \cite{Mar}. In some special cases, such as $n = 5$, the associated structure has been discussed by Dai and Li \cite{DL}.

In this work, we also consider the related properties of representation rings by means of cubic form expressions of generating relations, extending Witt's theory to the higher order, the so-called Harrison center theory \cite{Har}. Here we focus on the connection between the generating relations of the elements of the representation ring of the symmetric group $S_6$ and the resulting Harrison centers of cubic type. Finally, we obtain that the Harrison centers induced by the generating relations of the representation ring of $S_6$ are equivalent to themselves.

This paper is organised as follows: In Section \ref{2}, we first recall some basic contexts that are necessary for this work. In Section \ref{3}, starting from the irreducible characters of $S_6$, we 
give the power formulas of all irreducible representations of $S_6$ with the help of the recurrence relations of tensor product decomposition. 
In Section \ref{4}, we describe the complex representation ring $\mathcal{R}(S_6)$ and some of its properties, such as generators, defining relations, its unit group, primitive idempotents and its determinant and Casimir number. 
In Section \ref{5} we show that, for some certain dimensions $k$, $H_{f_{S_k}}$ must be isomorphic to Harrison centers $Z(f)$. 
Some open problems are left in Section \ref{6} for further study. 

\section{Basic representation theory of finite groups}\label{2}
In this section we will first recall some classical results in representation theory of finite groups which will be useful for our later calculations and proofs. We will refrain from giving a detailed introduction to these results, details of which can be found in \cite{Fe}. 

\begin{definition}
 The representation ring of the group, \( \mathcal{R}(G) \), is a unital commutative ring with \( \mathbb{Z} \)-basis \( \{[V_i] \mid 1 \leq i \leq s \} \), where additions and multiplications are determined by the following equation:
For any two complex representations \( [M],[N] \) of \( G \) then
\[
[M] + [N] = [M \oplus N], \quad [M] \cdot [N] = [M \otimes N], 
\]
where \( \{V_i \mid 1 \leq i \leq s\} \) is the complete set of mutually inequivalent irreducible complex representations of the group \(G \), and \( [M] \) is the isomorphic class corresponding to any finite dimensional complex representation \(M \) of \(G \). 
\end{definition}

\begin{lemma}\label{biaoshi_formula}
Let $G$ be a group and $\mathbb{F}$ be the algebraically closed field satisfying
$\text{char 
 }\mathbb {F} \nmid |G|$. Let $s$ be the number of conjugate classes of $G$, then
\begin{enumerate}
    \item $|\text{Irr}_{\mathbb{F}} G| = s$,
    \item $\text{Irr}_{\mathbb{F}} G = \{\chi_1,\ldots,\chi_s\}$, $n_i$ = \text{deg} $\chi_i$ then we have $$|G| = \sum_{i=1}^{s} n_i^2\text{}$$
\end{enumerate}

\end{lemma}

For a given symmetric group $S_n$ we want to find all irreducible representations, and the lemma \ref{biaoshi_formula} will help us to know if we have found all irreducible representations of $S_6$.

\begin{lemma}\label{huan_a}
If $\mathbb{F} = \mathbb{C}$.
$\text{Irr}_\mathbb{F} G = \{\chi_1,\ldots,\chi_s\}$, then
$$\frac{1}{|G|}\sum_{g \in G} \chi_i(g) \overline{\chi_j(g)} = \delta_{ij} = \begin{cases} 1, & i = j \\0. & i \neq j \end{cases} $$
\end{lemma}


Here it is not difficult to see that the methods for computing the characters of irreducible representations of the symmetric group $S_6$ are given by the lemma \ref{huan_a}, while the following lemma \ref{huan_c} can be used to quickly determine whether the two irreducible representation tensor products of the symmetric group $S_6$ are equivalent.

\begin{lemma}\label{huan_c}
Two complex representations of the group $G$ are equivalent if and only if they share the same characters.
\end{lemma}

\begin{lemma}\label{huan_d}
Let $(\rho,V)$ be a complex representation of a finite group $G$ and $V =\oplus_{i=1}^s n_iV_i$ be the decomposition of the direct sum of $G$, then
$$n_i = (\chi_i,\chi),$$
where $\chi$ is the character of $(\rho,V)$ and $\chi_i$ is the character of $(\rho_i,V_i)$.
\end{lemma}


\begin{lemma}\label{lemma_calculate}
Let $\{(\rho_i,V_i)|i=1,2,\cdots,s\}$ be a complete set of the mutually non-equivalent irreducible complex representations of the finite group $G$, and $V_i\otimes V_j\cong \oplus_{k=1}^s n_{ij}^k V_k$, for any $ 1\leq i,j\leq s$ be the decomposition of the tensor product of the irreducible complex representation of $G$,
$$
n_{ij}^k=\frac{1}{|G|}\sum_{t=1}^s |\mathcal{C}_t|\chi_{it}\chi_{jt}\overline{\chi_{kt}},
$$
where $\mathcal{C}_t$ is the conjugate class of the group $G$ and $\chi_{ij}$ is the value of its complex characters at position $(i,j)$.
\end{lemma}

\section{Power formulae for the symmetric group 
}\label{3}
Let us start from the complex irreducible characters of the symmetric group $S_6$ (See the specific value in the following table).
{\small
\begin{longtable}{|c|c|c|c|c|c|c|c|c|c|c|c|}
\hline
\diagbox{$Irr_{\mathbb C}S_6$}{$\xi$}                                                                     & \begin{tabular}[c]{@{}c@{}}$g_1$\\~ 1\end{tabular}                            & \begin{tabular}[c]{@{}c@{}}$g_2$\\~15\end{tabular}                              & \begin{tabular}[c]{@{}c@{}}$g_3$\\120\end{tabular}                             & \begin{tabular}[c]{@{}c@{}}$g_4$\\~15\end{tabular}                              & \begin{tabular}[c]{@{}c@{}}$g_5$\\~90\end{tabular}                             & \begin{tabular}[c]{@{}c@{}}$g_6$\\120\end{tabular}                            & \begin{tabular}[c]{@{}c@{}}$g_7$\\~45\end{tabular}                           & \begin{tabular}[c]{@{}c@{}}$g_8$\\~40\end{tabular}                            & \begin{tabular}[c]{@{}c@{}}$g_9$\\144\end{tabular}                           & \begin{tabular}[c]{@{}c@{}}$g_{10}$\\~90\end{tabular}                            & \begin{tabular}[c]{@{}c@{}}$g_{11}$\\~40\end{tabular}                             \endfirsthead
\hline
\begin{tabular}[c]{@{}c@{}}$\chi_1$\\ $\chi_2$\\ $\chi_3$\\ $\chi_4$\\ $\chi_5$\\ $\chi_6$\\ $\chi_7$\\ $\chi_8$\\ $\chi_9$\\ $\chi_{10}$\\ $\chi_{11}$\end{tabular} & \begin{tabular}[c]{@{}c@{}}1\\1\\5\\5\\9\\9\\10\\10\\5\\5\\16\end{tabular} & \begin{tabular}[c]{@{}c@{}}1\\-1\\3\\-3\\3\\-3\\2\\-2\\1\\-1\\0\end{tabular} & \begin{tabular}[c]{@{}c@{}}1~\\-1\\0\\0\\0\\0\\-1\\1\\1\\-1\\0\end{tabular} & \begin{tabular}[c]{@{}c@{}}1\\-1\\-1\\1\\3\\-3\\-2\\2\\-3\\3\\0\end{tabular} & \begin{tabular}[c]{@{}c@{}}1\\-1\\1\\-1\\-1\\1\\0\\0\\-1\\1\\0\end{tabular} & \begin{tabular}[c]{@{}c@{}}1\\-1\\-1\\1\\0\\0\\1\\-1\\0\\0\\0\end{tabular} & \begin{tabular}[c]{@{}c@{}}1\\1\\1\\1\\1\\1\\-2\\-2\\1\\1\\0\end{tabular} & \begin{tabular}[c]{@{}c@{}}1\\1\\2\\2\\0\\0\\1\\1\\-1\\-1\\-2\end{tabular} & \begin{tabular}[c]{@{}c@{}}1\\1\\0\\0\\-1\\-1\\0\\0\\0\\0\\1\end{tabular} & \begin{tabular}[c]{@{}c@{}}1\\1\\-1\\-1\\1\\1\\0\\0\\-1\\-1\\0\end{tabular} & \begin{tabular}[c]{@{}c@{}}1~\\1\\-1\\-1\\0\\0\\1\\1\\2\\2\\-2\end{tabular}  \\
\hline
\caption{The complex irreducible characters of $S_6$}
\end{longtable}}

With the help of the irreducible characters of $S_6$ above and the lemma~\ref{lemma_calculate}, we get the following equivalent relations, as follows.
\begin{theorem}\label{direct_composition}
    For any $i\in \{1,2,\cdots,11\}$, the direct sum decomposition of tensor products $V_1\otimes V_i$, $V_2\otimes V_i$ and $V_3\otimes V_i$ of different irreducible representations of $S_6$ is that 
    \begin{align*}
&V_1 \otimes V_j \cong V_j \otimes V_1 \cong V_j,\quad\text{for any $j\in \{1,2,3,4,5,6,7,8,9,10,11\}$},\\
&V_2\otimes V_i\cong  V_i\otimes V_2\cong\left\{
\begin{aligned}
    & V_{i-1},\quad\text{when $i\in \{2,4,6,8,10\}$}, \\
    & V_{i+1},\quad\text{when $i\in \{3,5,7,9\}$},
\end{aligned}
\right.\\
&V_2\otimes V_{11}\cong V_{11}\otimes V_2\cong V_{11},\\
&V_3\otimes V_3\cong V_1\oplus V_3\oplus V_5\oplus V_7,\\
&V_3\otimes V_4\cong V_4\otimes V_3\cong V_2\oplus V_4\oplus V_6\oplus V_8,\\
&V_3\otimes V_5\cong V_5\otimes V_3\cong V_3\oplus V_5\oplus V_7\oplus V_9\oplus V_{11},\\
&V_3\otimes V_6\cong V_6\otimes V_3\cong V_4\oplus V_6\oplus V_8\oplus V_{10}\oplus V_{11},\\
&V_3\otimes V_7\cong V_7\otimes V_3\cong V_3\oplus V_5\oplus V_7\oplus V_8\oplus V_{11},\\
&V_3\otimes V_8\cong V_8\otimes V_3\cong V_4\oplus V_6\oplus V_7\oplus V_8\oplus V_{11},\\
&V_3\otimes V_9\cong V_9\otimes V_3\cong V_5\oplus V_{11},\\
&V_3\otimes V_{10}\cong V_{10}\otimes V_3\cong V_6\oplus V_{11},\\
&V_3\otimes V_{11}\cong V_{11}\otimes V_3\cong V_5\oplus V_6\oplus V_7\oplus V_8\oplus V_9\oplus V_{10}\oplus 2V_{11}.
\end{align*}
\end{theorem}

\begin{proof}
By lemma~\ref{lemma_calculate}, it's obviously that the direct sum decomposition of $V_1\otimes V_i$ and $V_2\otimes V_i$ can be held. Let \begin{align*}
    V_3\otimes V_3 &\cong n_{33}^1V_1 \oplus n_{33}^2V_2 \oplus n_{33}^3V_3 \oplus n_{33}^3V_3 \oplus n_{33}^4V_4 \oplus n_{33}^5V_5 \oplus n_{33}^6V_6 \\
    &\oplus n_{33}^ 7V_7 \oplus n_{33}^8V_8 \oplus n_{33}^9V_9 \oplus n_{33}^{10}V_{10} \oplus n_{33}^{11}V_{11},
    \end{align*} then we have $n_{33}^1=n_{33}^3=n_{33}^5=n_{33}^7=1,n_{ 33}^2=n_{33}^4=n_{33}^6=n_{33}^8=n_{33}^9=n_{33}^{10}=n_{33}^{11}=0$. It implies $V_3\otimes V_3\cong V_1\oplus V_3\oplus V_5\oplus V_7$, and similarly
    \begin{align*}
&V_3\otimes V_4\cong V_4\otimes V_3\cong V_2\oplus V_4\oplus V_6\oplus V_8\\
&V_3\otimes V_5\cong V_5\otimes V_3\cong V_3\oplus V_5\oplus V_7\oplus V_9\oplus V_{11}\\
&V_3\otimes V_6\cong V_6\otimes V_3\cong V_4\oplus V_6\oplus V_8\oplus V_{10}\oplus V_{11}\\
&V_3\otimes V_7\cong V_7\otimes V_3\cong V_3\oplus V_5\oplus V_7\oplus V_8\oplus V_{11}\\
&V_3\otimes V_8\cong V_8\otimes V_3\cong V_4\oplus V_6\oplus V_7\oplus V_8\oplus V_{11}\\
&V_3\otimes V_9\cong V_9\otimes V_3\cong V_5\oplus V_{11}\\
&V_3\otimes V_{10}\cong V_{10}\otimes V_3\cong V_6\oplus V_{11}\\
&V_3\otimes V_{11}\cong V_{11}\otimes V_3\cong V_5\oplus V_6\oplus V_7\oplus V_8\oplus V_9\oplus V_{10}\oplus 2V_{11}
\end{align*}
\end{proof}

According to the idea of theorem~\ref{direct_composition}, for any $i\in \{1,2,\cdots,11\}$, we can further obtain the direct sum decomposition for the remaining tensor products $V_5\otimes V_i,V_6\otimes V_i,\cdots,V_{11}\otimes V_i$ (see Appendix I for more details).

Then we will give the power formula of the irreducible representation of $S_6$ for $V_3,V_4,\cdots,V_{11}$.

\begin{theorem}\label{formula}
    The power formula of the irreducible representation of $S_6$ for $V_3$ is as follows
    $$
    V_3^{\otimes n}= \oplus_{i=1}^{11}{ \frac{c_{i1}\cdot(-1)^n+c_{i2}\cdot 2^n+c_{i3}\cdot 3^n+c_{i4}\cdot 5^n+c_{i5}}{c_{i6}}V_i},
    $$
    where $c_{ij}$ is the $(i,j)$ position in the following table that corresponds to $V_i$. For example, $c_{23}$=-15, and the specific coefficients can be found in the following table.
    
\begin{longtable}{|c|c|c|c|c|c|c|}
\toprule
\diagbox{Bas.}{Coe.} & $c_{i1}$  & $c_{i2}$ & $c_{i3}$ & $c_{i4}$ & $c_{i5}$  & $c_{i6}$ \\
\hline
\endfirsthead
\toprule
\diagbox{Bas.}{Coe.} & $c_{i1}$  & $c_{i2}$ & $c_{i3}$ & $c_{i4}$ & $c_{i5}$  & $c_{i6}$ \\
\hline
\endhead
\hline
$V_1$    & $265$  & $40$  & $15$  & $1$   & $135$ & $720$ \\
\hline
$V_2$    & $-5$   & $40$  & $-15$ & $1$   & $-45$ & $720$ \\
\hline
$V_3$    & $-53$  & $16$  & $9$   & $1$   & $27$  & $144$ \\
\hline
$V_4$    & $1$    & $16$  & $-9$  & $1$   & $-9$  & $144$ \\
\hline
$V_5$    & $15$   & $0$   & $5$   & $1$   & $-5$  & $80$ \\
\hline
$V_6$    & $5$    & $0$   & $-5$  & $1$   & $15$  & $80$ \\
\hline
$V_7$    & $13$   & $4$   & $3$   & $1$   & $-9$  & $72$ \\
\hline
$V_8$    & $-5$   & $4$   & $-3$  & $1$   & $-9$  & $72$ \\
\hline
$V_9$    & $-11$  & $-8$  & $3$   & $1$   & $-9$  & $144$ \\
\hline
$V_{10}$ & $7$    & $-8$  & $-3$  & $1$   & $27$  & $144$ \\
\hline
$V_{11}$ & $-5$   & $-5$  & $0$   & $1$   & $0$   & $45$ \\
\bottomrule

\caption{Coefficients of the power formula of $V_3$} \label{co4}
\end{longtable}

For example, the coefficient of the term $V_4$ of $V_3^{\otimes n}$ is 
$$
\frac{(-1)^n+16\cdot 2^n-9\cdot 3^n+5^n-9}{144}.
$$

The power formula for $V_4$ is
    $$
    V_4^{\otimes n}=\oplus_{i=1}^{11}{ \frac{c_{i1}\cdot(-1)^n+c_{i2}\cdot 2^n+c_{i3}\cdot (-1)^n\cdot 3^n+c_{i4}\cdot 5^n+c_{i5} }{c_{i6}}V_i}.
    $$
    where $c_{ij}$ is found in the table below.
\begin{longtable}{|c|c|c|c|c|c|c|}
\toprule
\diagbox{Bas.}{Coe.} &$c_{i1}$  & $c_{i2}$ & $c_{i3}$ & $c_{i4}$ & $c_{i5}$  & $c_{i6}$ \\
\hline
\endfirsthead
\toprule
\diagbox{Bas.}{Coe.} & $c_{i1}$  & $c_{i2}$ & $c_{i3}$ & $c_{i4}$ & $c_{i5}$  & $c_{i6}$ \\
\hline
\endhead
\hline
$V_1$    & $220$  & $40$  & $15$  & $1$   & $180$ & $720$ \\
\hline
$V_2$    & $40$   & $40$  & $-15$ & $1$   & $-90$ & $720$ \\
\hline
$V_3$    & $-8$  & $16$  & $9$   & $1$   & $-18$  & $144$ \\
\hline
$V_4$    & $-44$    & $16$  & $-9$  & $1$   & $36$  & $144$ \\
\hline
$V_5$    & $0$   & $0$   & $5$   & $1$   & $10$  & $80$ \\
\hline
$V_6$    & $20$    & $0$   & $-5$  & $1$   & $0$  & $80$ \\
\hline
$V_7$    & $4$   & $4$   & $3$   & $1$   & $0$  & $72$ \\
\hline
$V_8$    & $4$   & $4$   & $-3$  & $1$   & $-18$  & $72$ \\
\hline
$V_9$    & $-20$  & $-8$  & $3$   & $1$   & $0$  & $144$ \\
\hline
$V_{10}$ & $16$    & $-8$  & $-3$  & $1$   & $18$  & $144$ \\
\hline
$V_{11}$ & $-5$   & $-5$  & $0$   & $1$   & $0$   & $45$ \\
\bottomrule
\caption{Coefficients of power formula of $V_4$}
\end{longtable}
The power formula for the irreducible representation $V_5$ is 
$$
    V_5^{\otimes n}= \oplus_{i=1}^{11}{ \frac{c_{i1}\cdot(-1)^n+c_{i2}\cdot 3^{(2n)}+c_{i3}\cdot 3^n+c_{i4}}{c_{i5}}V_i},
    $$
\begin{longtable}{|c|c|c|c|c|c|}
\toprule
\diagbox{Bas.}{Coe.} & $c_{i1}$  & $c_{i2}$ & $c_{i3}$ & $c_{i4}$  & $c_{i5}$ \\
\hline
\endfirsthead
\toprule
\diagbox{Bas.}{Coe.} & $c_{i1}$  & $c_{i2}$ & $c_{i3}$ & $c_{i4}$  & $c_{i5}$ \\
\hline
\endhead
\hline
$V_1$    & $234$  & $1$  & $30$  & $135$    & $720$ \\
\hline
$V_2$    & $54$   & $1$  & $-30$ & $135$    & $720$ \\
\hline
$V_3$    & $18$  & $1$  & $6$   & $-9$     & $144$ \\
\hline
$V_4$    & $-18$    & $1$  & $-6$  & $-9$    & $144$ \\
\hline
$V_5$    & $-26$   & $1$   & $10$   & $15$   & $80$ \\
\hline
$V_6$    & $-6$    & $1$   & $-10$  & $15$     & $80$ \\
\hline
$V_7$    & $0$   & $1$   & $0$   & $-9$     & $72$ \\
\hline
$V_8$    & $0$   & $1$   & $0$  & $-9$     & $72$ \\
\hline
$V_9$    & $-18$  & $1$  & $-6$   & $-9$    & $144$ \\
\hline
$V_{10}$ & $18$    & $1$  & $6$  & $-9$     & $144$ \\
\hline
$V_{11}$ & $9$   & $1$  & $0$   & $0$     & $45$ \\
\hline
\caption{Table of coefficients of power formula of $V_5$}
\end{longtable}
The power formula for the irreducible representation $V_6$ is 
$$
    V_6^{\otimes n}= \oplus_{i=1}^{11}{ \frac{c_{i1}\cdot(-1)^n+c_{i2}\cdot 3^{(2n)}+c_{i3}\cdot (-1)^n\cdot 3^n+c_{i4}}{c_{i5}}V_i},
    $$

\begin{longtable}{|c|c|c|c|c|c|}
\toprule
\diagbox{Bas.}{Coe.} & $c_{i1}$  & $c_{i2}$ & $c_{i3}$ & $c_{i4}$  & $c_{i5}$ \\
\hline
\endfirsthead
\toprule
\diagbox{Bas.}{Coe.} & $c_{i1}$  & $c_{i2}$ & $c_{i3}$ & $c_{i4}$  & $c_{i5}$ \\
\hline
\endhead
\hline
$V_1$    & $144$  & $1$  & $30$  & $225$    & $720$ \\
\hline
$V_2$    & $144$   & $1$  & $-30$ & $45$    & $720$ \\
\hline
$V_3$    & $0$  & $1$  & $6$   & $9$     & $144$ \\
\hline
$V_4$    & $0$    & $1$  & $-6$  & $-27$    & $144$ \\
\hline
$V_5$    & $-16$   & $1$   & $10$   & $5$   & $80$ \\
\hline
$V_6$    & $-16$    & $1$   & $-10$  & $25$     & $80$ \\
\hline
$V_7$    & $0$   & $1$   & $0$   & $-9$     & $72$ \\
\hline
$V_8$    & $0$   & $1$   & $0$  & $-9$     & $72$ \\
\hline
$V_9$    & $0$  & $1$  & $-6$   & $-27$    & $144$ \\
\hline
$V_{10}$ & $0$    & $1$  & $6$  & $9$     & $144$ \\
\hline
$V_{11}$ & $9$   & $1$  & $0$   & $0$     & $45$ \\
\bottomrule
\caption{Coefficients of power formula of $V_6$}
\end{longtable}
The power formula for the irreducible representation $V_7$ is
$$
    V_7^{\otimes n}= \oplus_{i=1}^{11}{ \frac{c_{i1}\cdot(-1)^n+c_{i2}\cdot (-1)^n\cdot 2^n+c_{i3}\cdot 2^n+c_{i4}\cdot 2^n\cdot 5^n+c_{i5}}{c_{i6}}V_i},
    $$
\begin{longtable}{|c|c|c|c|c|c|c|}
\toprule
\diagbox{Bas.}{Coe.} & $c_{i1}$  & $c_{i2}$ & $c_{i3}$ & $c_{i4}$ & $c_{i5}$  & $c_{i6}$  \\
\hline
\endfirsthead
\toprule
\diagbox{Bas.}{Coe.} & $c_{i1}$  & $c_{i2}$ & $c_{i3}$ & $c_{i4}$ & $c_{i5}$  & $c_{i6}$  \\
\hline
\endhead
\hline
$V_1$    & $120$  & $60$  & $15$  & $1$   & $200$ & $720$ \\
\hline
$V_2$    & $-120$   & $30$  & $-15$ & $1$   & $-40$ & $720$ \\
\hline
$V_3$    & $0$  & $6$  & $9$   & $1$   & $-16$  & $144$ \\
\hline
$V_4$    & $0$    & $12$  & $-9$  & $1$   & $32$  & $144$ \\
\hline
$V_5$    & $0$   & $10$   & $5$   & $1$   & $0$  & $80$ \\
\hline
$V_6$    & $0$    & $0$   & $-5$  & $1$   & $0$  & $80$ \\
\hline
$V_7$    & $-12$   & $-12$   & $3$   & $1$   & $20$  & $72$ \\
\hline
$V_8$    & $12$   & $-6$   & $-3$  & $1$   & $-4$  & $72$ \\
\hline
$V_9$    & $24$  & $0$  & $3$   & $1$   & $8$  & $144$ \\
\hline
$V_{10}$ & $-24$    & $18$  & $-3$  & $1$   & $8$  & $144$ \\
\hline
$V_{11}$ & $0$   & $0$  & $0$   & $1$   & $-10$   & $45$ \\
\bottomrule
\caption{Coefficients of power formula of $V_7$}
\end{longtable}
The power formula for the irreducible representation $V_8$ is
$$
    V_8^{\otimes n}= \oplus_{i=1}^{11}{ \frac{c_{i1}\cdot(-1)^n+c_{i2}\cdot (-1)^n\cdot 2^n+c_{i3}\cdot 2^n+c_{i4}\cdot 2^n\cdot 5^n+c_{i5}}{c_{i6}}V_i},
    $$

\begin{longtable}{|c|c|c|c|c|c|c|}
\toprule
\diagbox{Bas.}{Coe.} & $c_{i1}$  & $c_{i2}$ & $c_{i3}$ & $c_{i4}$ & $c_{i5}$  & $c_{i6}$  \\
\hline
\endfirsthead
\toprule
\diagbox{Bas.}{Coe.} & $c_{i1}$  & $c_{i2}$ & $c_{i3}$ & $c_{i4}$ & $c_{i5}$  & $c_{i6}$  \\
\hline
\endhead
\hline
$V_1$    & $120$  & $60$  & $15$  & $1$   & $200$ & $720$ \\
\hline
$V_2$    & $-120$   & $30$  & $-15$ & $1$   & $-40$ & $720$ \\
\hline
$V_3$    & $-24$  & $18$  & $-3$   & $1$   & $8$  & $144$ \\
\hline
$V_4$    & $24$    & $0$  & $3$  & $1$   & $8$  & $144$ \\
\hline
$V_5$    & $0$   & $10$   & $5$   & $1$   & $0$  & $80$ \\
\hline
$V_6$    & $0$    & $0$   & $-5$  & $1$   & $0$  & $80$ \\
\hline
$V_7$    & $12$   & $-6$   & $-3$   & $1$   & $-4$  & $72$ \\
\hline
$V_8$    & $-12$   & $-12$   & $3$  & $1$   & $20$  & $72$ \\
\hline
$V_9$    & $0$  & $12$  & $-9$   & $1$   & $32$  & $144$ \\
\hline
$V_{10}$ & $0$    & $6$  & $9$  & $1$   & $-16$  & $144$ \\
\hline
$V_{11}$ & $0$   & $0$  & $0$   & $1$   & $-10$   & $45$ \\
\bottomrule
\caption{Coefficients of power formula of $V_8$}
\end{longtable}
The power formula for the irreducible representation $V_9$ is 
$$
    V_9^{\otimes n}=\oplus_{i=1}^{11}{ \frac{c_{i1}\cdot(-1)^n+c_{i2}\cdot 2^n+c_{i3}\cdot (-1)^n\cdot 3^n+c_{i4}\cdot 5^n+c_{i5}}{c_{i6}}V_i},
    $$
\begin{longtable}{|c|c|c|c|c|c|c|}
\toprule
\diagbox{Bas.}{Coe.} & $c_{i1}$  & $c_{i2}$ & $c_{i3}$ & $c_{i4}$ & $c_{i5}$  & $c_{i6}$  \\
\hline
\endfirsthead
\toprule
\diagbox{Bas.}{Coe.} & $c_{i1}$  & $c_{i2}$ & $c_{i3}$ & $c_{i4}$ & $c_{i5}$  & $c_{i6}$  \\
\hline
\endhead
\hline
$V_1$    & $220$  & $40$  & $15$  & $1$   & $180$ & $720$ \\
\hline
$V_2$    & $40$   & $40$  & $-15$ & $1$   & $-90$ & $720$ \\
\hline
$V_3$    & $16$  & $-8$  & $-3$   & $1$   & $18$  & $144$ \\
\hline
$V_4$    & $-20$    & $-8$  & $3$  & $1$   & $0$  & $144$ \\
\hline
$V_5$    & $0$   & $0$   & $5$   & $1$   & $10$  & $80$ \\
\hline
$V_6$    & $20$    & $0$   & $-5$  & $1$   & $0$  & $80$ \\
\hline
$V_7$    & $4$   & $4$   & $-3$   & $1$   & $-18$  & $72$ \\
\hline
$V_8$    & $4$   & $4$   & $3$  & $1$   & $0$  & $72$ \\
\hline
$V_9$    & $-44$  & $16$  & $-9$   & $1$   & $36$  & $144$ \\
\hline
$V_{10}$ & $-8$    & $16$  & $9$  & $1$   & $-18$  & $144$ \\
\hline
$V_{11}$ & $-5$   & $-5$  & $0$   & $1$   & $0$   & $45$ \\
\bottomrule
\caption{Coefficients of power formula of $V_9$}
\end{longtable}
The power formula for the irreducible representation $V_{10}$ is
$$
    V_{10}^{\otimes n}= \oplus_{i=1}^{11}{ \frac{c_{i1}\cdot(-1)^n+c_{i2}\cdot 2^n+c_{i3}\cdot 3^n+c_{i4}\cdot 5^n+c_{i5}}{c_{i6}}V_i},
    $$

\begin{longtable}{|c|c|c|c|c|c|c|}
\toprule
\diagbox{Bas.}{Coe.} & $c_{i1}$  & $c_{i2}$ & $c_{i3}$ & $c_{i4}$ & $c_{i5}$  & $c_{i6}$  \\
\hline
\endfirsthead
\toprule
\diagbox{Bas.}{Coe.} & $c_{i1}$  & $c_{i2}$ & $c_{i3}$ & $c_{i4}$ & $c_{i5}$  & $c_{i6}$  \\
\hline
\endhead
\hline
$V_1$    & $265$  & $40$  & $15$  & $1$   & $135$ & $720$ \\
\hline
$V_2$    & $-5$   & $40$  & $-15$ & $1$   & $-45$ & $720$ \\
\hline
$V_3$    & $7$  & $-8$  & $-3$   & $1$   & $27$  & $144$ \\
\hline
$V_4$    & $-11$    & $-8$  & $3$  & $1$   & $-9$  & $144$ \\
\hline
$V_5$    & $15$   & $0$   & $5$   & $1$   & $-5$  & $80$ \\
\hline
$V_6$    & $5$    & $0$   & $-5$  & $1$   & $15$  & $80$ \\
\hline
$V_7$    & $-5$   & $4$   & $-3$   & $1$   & $-9$  & $72$ \\
\hline
$V_8$    & $13$   & $4$   & $3$  & $1$   & $-9$  & $72$ \\
\hline
$V_9$    & $1$  & $16$  & $-9$   & $1$   & $-9$  & $144$ \\
\hline
$V_{10}$ & $-53$    & $16$  & $9$  & $1$   & $27$  & $144$ \\
\hline
$V_{11}$ & $-5$   & $-5$  & $0$   & $1$   & $0$   & $45$ \\
\bottomrule
\caption{Coefficients of power formula of $V_{10}$}
\end{longtable}
The power formula for the irreducible representation $V_{11}$ is 
$$
    V_{11}^{\otimes n}= \oplus_{i=1}^{11}{ \frac{c_{i1}\cdot 2^{(4n)}+c_{i2}\cdot (-1)^n\cdot 2^n+c_{i3}}{c_{i4}}V_i},
    $$
\begin{longtable}{|c|c|c|c|c|}
\toprule
\diagbox{Bas.}{Coe.} & $c_{i1}$  & $c_{i2}$ & $c_{i3}$  & $c_{i4}$ \\
\hline
\endfirsthead
\toprule
\diagbox{Bas.}{Coe.} & $c_{i1}$  & $c_{i2}$ & $c_{i3}$  & $c_{i4}$ \\
\hline
\endhead
\hline
$V_1$    & $1$  & $80$  & $144$   & $720$ \\
\hline
$V_2$    & $1$   & $80$  & $144$  & $720$ \\
\hline
$V_3$    & $1$  & $8$  & $0$    & $144$ \\
\hline
$V_4$    & $1$    & $8$  & $0$  & $144$ \\
\hline
$V_5$    & $1$   & $0$   & $-16$   & $80$ \\
\hline
$V_6$    & $1$    & $0$   & $-16$ & $80$ \\
\hline
$V_7$    & $1$   & $8$   & $0$     & $72$ \\
\hline
$V_8$    & $1$   & $8$   & $0$   & $72$ \\
\hline
$V_9$    & $1$  & $8$  & $0$   & $144$ \\
\hline
$V_{10}$ & $1$    & $8$  & $0$    & $144$ \\
\hline
$V_{11}$ & $1$   & $-10$  & $9$   & $45$ \\
\bottomrule
\caption{Coefficients of power formula of $V_{11}$}
\end{longtable}
\end{theorem}

\begin{proof}
We show the case $V_3$ as an example. Let us continue with
    $$
V_3^{\otimes n} \cong a_n V_1 \oplus b_n V_2 \oplus c_n V_3 \oplus d_n V_4 \oplus e_n V_5 \oplus f_n V_6 \oplus g_n V_7 \oplus h_n V_8 \oplus i_n V_9 \oplus j_n V_{10} \oplus k_n V_{11}
    $$
    Then
    \begin{align*}
V_3^{\otimes n} &\cong V_3^{\otimes (n-1)}\otimes V_3\\
&\cong  a_{n-1} (V_1\otimes V_3) \oplus b_{n-1} (V_2\otimes V_3) \oplus c_{n-1} (V_3\otimes V_3) \oplus d_{n-1} (V_4\otimes V_3)\\
 &\oplus e_{n-1} (V_5\otimes V_3) \oplus f_{n-1} (V_6\otimes V_3) \oplus g_{n-1} (V_7\otimes V_3) \oplus h_{n-1} (V_8\otimes V_3)\\
  &\oplus i_{n-1} (V_9\otimes V_3) \oplus j_{n-1} (V_{10}\otimes V_3) \oplus k_{n-1} (V_{11}\otimes V_3)\\
  &\cong a_{n-1} V_3 \oplus b_{n-1} V_4 \oplus c_{n-1} (V_1\oplus V_3\oplus V_5\oplus V_7) \oplus d_{n-1} (V_2\oplus V_4\oplus V_6\oplus V_8)\\
  &\oplus e_{n-1} (V_3\oplus V_5\oplus V_7\oplus V_9\oplus V_{11})\oplus f_{n-1} (V_4\oplus V_6\oplus V_8\oplus V_{10}\oplus V_{11}) \\
  &\oplus g_{n-1} (V_3\oplus V_5\oplus V_7\oplus V_8\oplus V_{11}) \oplus h_{n-1} (V_4\oplus V_6\oplus V_7\oplus V_8\oplus V_{11})\\
  &\oplus i_{n-1} (V_5\oplus V_{11}) \oplus j_{n-1} (V_6\oplus V_{11}) \oplus k_{n-1} (V_5\oplus V_6\oplus V_7\oplus V_8\oplus V_9\oplus V_{10}\oplus 2V_{11})
    \end{align*}
    These terms can be rewritten:
\[
\left\{
    \begin{aligned}
        & a_n=c_{n-1}\\
        & b_n=d_{n-1}\\
        & c_n=a_{n-1}+c_{n-1}+e_{n-1}+g_{n-1}\\
        & d_n=b_{n-1}+d_{n-1}+f_{n-1}+h_{n-1}\\
        & e_n=c_{n-1}+e_{n-1}+g_{n-1}+i_{n-1}+k_{n-1}\\
        & f_n=d_{n-1}+f_{n-1}+h_{n-1}+j_{n-1}+k_{n-1}\\
        & g_n=c_{n-1}+e_{n-1}+g_{n-1}+h_{n-1}+k_{n-1}\\
        & h_n=d_{n-1}+f_{n-1}+g_{n-1}+h_{n-1}+k_{n-1}\\
        & i_n=e_{n-1}+k_{n-1}\\
        & j_n=f_{n-1}+k_{n-1}\\
        & k_n=e_{n-1}+f_{n-1}+g_{n-1}+h_{n-1}+i_{n-1}+j_{n-1}+2k_{n-1}
    \end{aligned}
\right.
\]
In other words, we have $T_n = AT_{n-1}$, where 
\begin{align*}
    T_n = \begin{pmatrix}
a_n \\ b_n \\ c_n \\ d_n \\e_n \\ f_n \\\ g_n \\ h_n \\ i_n \\ j_n \\ k_n
\end{pmatrix},\quad A=\begin{pmatrix}
0 & 0 & 1 & 0 & 0 & 0 & 0 & 0 & 0 & 0 \quad 0 \\
0 & 0 & 0 & 1 & 0 & 0 & 0 & 0 & 0 & 0 \quad 0 \\
1 & 0 & 1 & 0 & 1 & 0 & 1 & 0 & 0 & 0 \quad 0 \\
0 & 1 & 0 & 1 & 0 & 1 & 0 & 1 & 0 & 0 \quad 0 \\
0 & 0 & 1 & 0 & 1 & 0 & 1 & 0 & 1 & 0 \quad 1 \\
0 & 0 & 0 & 1 & 0 & 1 & 0 & 1 & 0 & 1 \quad 1 \\
0 & 0 & 1 & 0 & 1 & 0 & 1 & 1 & 0 & 0 \quad 1 \\
0 & 0 & 0 & 1 & 0 & 1 & 1 & 1 & 0 & 0 \quad 1 \\
0 & 0 & 0 & 0 & 1 & 0 & 0 & 0 & 0 & 0 \quad 1 \\
0 & 0 & 0 & 0 & 0 & 1 & 0 & 0 & 0 & 0 \quad 1 \\
0 & 0 & 0 & 0 & 1 & 1 & 1 & 1 & 1 & 1 \quad 2
\end{pmatrix}, \quad T_{n-1}=\begin{pmatrix}
a_{n-1} \\b_{n-1} \\c_{n-1}\\d_{n-1}\\e_{n-1}\\f_{n-1}\\g_{n-1}\\h_{n-1}\\i_{n-1}\\j_{n-1}\\ k_{n-1}
\end{pmatrix}.
\end{align*}

By iteration we get $T_n = A^{n-1}T_1$, so the general formula for finding $a_n$ to $k_n$ is transformed into a problem of finding the $(n-1)$-th power of $A$.
Note that the characteristic matrix of $A$ is
\begin{center}
\renewcommand{\arraystretch}{1.5} 
$\left(
\begin{tabular}{*{11}{c}}
$-\lambda$ & 0 & 1 & 0 & 0 & 0 & 0 & 0 & 0 & 0 & 0 \\
0 & $-\lambda$ & 0 & 1 & 0 & 0 & 0 & 0 & 0 & 0 & 0 \\
1 & 0 & $1-\lambda$ & 0 & 1 & 0 & 1 & 0 & 0 & 0 & 0 \\
0 & 1 & 0 & $1-\lambda$ & 0 & 1 & 0 & 1 & 0 & 0 & 0 \\
0 & 0 & 1 & 0 & $1-\lambda$ & 0 & 1 & 0 & 1 & 0 & 1 \\
0 & 0 & 0 & 1 & 0 & $1-\lambda$ & 0 & 1 & 0 & 1 & 1 \\
0 & 0 & 1 & 0 & 1 & 0 & $1-\lambda$ & 1 & 0 & 0 & 1 \\
0 & 0 & 0 & 1 & 0 & 1 & 1 & $1-\lambda$ & 0 & 0 & 1 \\
0 & 0 & 0 & 0 & 1 & 0 & 0 & 0 & $-\lambda$ & 0 & 1 \\
0 & 0 & 0 & 0 & 0 & 1 & 0 & 0 & 0 & $-\lambda$ & 1 \\
0 & 0 & 0 & 0 & 1 & 1 & 1 & 1 & 1 & 1 & $2-\lambda$\\
\end{tabular}
\right)$
\end{center}
Thus, all the eigenvalues of $A$ can be directly computed: $\lambda_1=0$, $\lambda_2=0$, $\lambda_3=1$, $\lambda_4=1$, $\lambda_5=-1$, $\lambda_6=-1$, $\lambda_7=-1$, $\lambda_8= -1$, $\lambda_9=2$, $\lambda_{10}=3$, $\lambda_{11}=5$. 

Since the matrix $A$ is diagonal, it is similarly diagonalisable. Then there exists an invertible matrix $P$ such that $P^{-1}AP=\Lambda$ is diagonal. The matrices $\Lambda$ and $P$ have the form

$$
\Lambda =\begin{pmatrix}
0 & 0 & 0 & 0 & 0 & 0 & 0 & 0 & 0 & 0 \quad 0 \\
0 & 0 & 0 & 0 & 0 & 0 & 0 & 0 & 0 & 0 \quad 0 \\
0 & 0 & 1 & 0 & 0 & 0 & 0 & 0 & 0 & 0 \quad 0 \\
0 & 0 & 0 & 1 & 0 & 0 & 0 & 0 & 0 & 0 \quad 0 \\
0 & 0 & 0 & 0 & -1 & 0 & 0 & 0 & 0 & 0 \quad 0 \\
0 & 0 & 0 & 0 & 0 & -1 & 0 & 0 & 0 & 0 \quad 0 \\
0 & 0 & 0 & 0 & 0 & 0 & -1 & 0 & 0 & 0 \quad 0 \\
0 & 0 & 0 & 0 & 0 & 0 & 0 & -1 & 0 & 0 \quad 0 \\
0 & 0 & 0 & 0 & 0 & 0 & 0 & 0 & 2 & 0 \quad 0 \\
0 & 0 & 0 & 0 & 0 & 0 & 0 & 0 & 0 & 3 \quad 0 \\
0 & 0 & 0 & 0 & 0 & 0 & 0 & 0 & 0 & 0  \quad 5
\end{pmatrix}
$$
\begin{center}
\renewcommand{\arraystretch}{1.5} 
$P=\frac{1}{16} \left(
\begin{tabular}{*{11}{c}}
$-16$ & 16 & 0 & 16 & $-16$ & $-16$ & 0 & $-32$& $-8$ & $-16$& 1 \\
16 & 16 & 16 & 0 & 16 & 0 & $-16$ & $-16$ & $-8$ & 16& 1 \\
0 & 0 & 0 & 16 & 16 & 16 & 0 & 32 & $-16$ & $-48$ & 5 \\
0 & 0 & 16 & 0 & $-16$ & 0 & 16 & 16 & $-16$ & 48 & 5 \\
0 & $-16$ & 16 & 0 & 0 & $-16$ & 0 & $-16$ & 0 & $-48$ & 9 \\
0 & $-16$ & 0 & 16 & 0 & 0 & $-16$ & $-16$ & 0 & 48 & 9 \\
16 & 0 & $-16$ & $-16$ & $-16$ & 0 & 0 & $-16$ & $-8$ & $-32$ & 10 \\
16 & 0 & $-16$ & $-16$ & 16 & 0 & 0 & 0 & $-8$ & 32 & 10 \\
$-16$ & 0 & 16 & 0 & 0 & 16 & 0 & 0 & 8 & $-16$ & 5 \\
16 & 0 & 0 & 16 & 0 & 0 & 16 & 0 & 8 & 16 & 5 \\
0 & 16 & 0 & 0 & 0 & 0 & 0 & 16 & 16 & 0 & 16 \\
\end{tabular}
\right)$,
\end{center}

By calculation we get $P^{-1}$.
\begin{center}
\renewcommand{\arraystretch}{1.5} 
$ \frac{1}{720}\left(
\begin{tabular}{*{11}{c}}
$-120$ & 120 & 0 & 0 & 0 & 0 & 120 & $-120$ & $-120$ & 120 & 0 \\
144 & 144 & 0 & 0 & $-144$ & $-144$ & 0 & 0 & 0 & 0 & 144 \\
$-45$ & 135 & $-45$ & 135 & 135 & $-45$ & $-90$ & $-90$ & 135 & $-45$ & 0 \\
135 & $-45$ & 135 & $-45$ & $-45$ & 135 & $-90$ & $-90$ & $-45$ & 135 & 0 \\
$-50$ & 130 & 50 & $-130$ & 90 & $-90$ & $-140$ & 220 & $-10$ & 170 & $-80$ \\
$-55$ & 35 & 55 & $-35$ & $-225$ & 45 & 170 & $-10$ & 385 & 115 & $-160$ \\
35 & $-55$ & $-35$ & 55 & 45 & $-225$ & $-10$ & 170 & 115 & 385 & $-160$ \\
$-80$ & $-80$ & 80 & 80 & 0 & 0 & $-80$ & $-80$ & $-160$ & $-160$ & 160 \\
$-80$ & $-80$ & $-160$ & $-160$ & 0 & 0 & $-80$ & $-80$ & 80 & 80 & 160 \\
$-15$ & 15 & $-45$ & 45 & $-45$ & 45 & $-30$ & 30 & $-15$ & 15 & 0 \\
16 & 16 & 80 & 80 & 144 & 144 & 160 & 160 & 80 & 80 & 256 \\
\end{tabular}
\right)$,
\end{center}

Let $A^{n-1}=(a_{ij})_{n\times n}$, then we have $A^{n-1}=P\Lambda^{n-1}P^{-1}$. To get $T_n$. Notice that 
$$
T_1=\begin{pmatrix}
0 & 0& 1&0 &0 &0&0 &0 &0&0 \quad0
\end{pmatrix}^T.
$$
From the equation $T_n = A^{n-1}T_1$ we can obtain the coefficients (see Table~\ref{co4}) we want. Similarly, we can also get the power formulas for the remaining irreducible representations $V_4,V_5,\cdots,V_{11}$.
\end{proof}


\section{Properties of the representation ring $\mathcal{R}(S_6)$}\label{4}
To understand the properties of the representation ring $\mathcal{R}(S_6)$, we define the maps $\Gamma:\mathcal{R}(S_6)\rightarrow \mathbb{Z}$:
$$[V_1]\mapsto 1, [V_2]\mapsto x_1, [V_3]\mapsto x_2,\cdots, [V_{11}]\mapsto x_{10}.$$

Then we can use the generators to express the decomposition of direct sums, i.e. we can rewrite $V_2\otimes V_2\cong V_1$ as $x_1^2-1=0$. Then we have the following theorem.

\begin{theorem}\label{shengcheng}
The representation ring $\mathcal{R}(S_6)$ can be expressed by the following generators:
\begin{align*}
\mathcal{R}(S_6)&\cong \mathbb Z[x_1,x_2,\cdots,x_{10}]/(y_1,y_2,\cdots,y_{55}),
\end{align*}
where $y_1,y_2,\cdots,y_{55}$ are generated by $x_1,x_2,\cdots,x_{10}$. The expressions of $y_i$ are listed in appendix II (\ref{fulub}).
\end{theorem}

\begin{proof}
    Just like the example we showed before, by writing $V_2\otimes V_2\cong V_1$ in the form of $x_1^2-1=0$, the whole relations are not hard to achieve.
\end{proof}

\begin{theorem}\label{group}
  The unit group of $\mathcal{R}(S_6)$, denoted as $U(\mathcal{R}(S_6))$, is isomorphic to Klein four-group.
\end{theorem}

\begin{proof}
For any finite group $ G $, denote its complex representation ring as $ \mathcal{R}(G) $. Let $ g \in \mathcal{C} $ where $\mathcal{C}$ is  the conjugate class of $ G $. 
We can define the homomorphisms
\[ \varphi_\mathcal{C} :  \mathcal{R}(G)  \to \mathcal{O}_\mathbb{C}, \quad [V] \mapsto \chi_{V(g)} \]
It is easy to see that $ \varphi_\mathcal{C} $ is a ring homomorphism mapping from the representation ring $ \mathcal{R}(G) $ to the ring of algebraic integers $\mathcal{O}_\mathbb{C}$ which implies that it maps any unit in $ \mathcal{R}(G) $ to the unit in $\mathcal{O}_\mathbb{C}$. In particular, since the characters of the symmetric groups are all integers, we have the ring homomorphisms, for $1 \leq i \leq 11$,
\[ \varphi_i :  \mathcal{R}(S_6)  \to \mathbb{Z}, \quad [V_j] \mapsto \chi_{ji}\]

For any unit of $ \mathcal{R}(S_6) $,  the vector $ (\varphi_1(u), \ldots, \varphi_{11}(u))^T $ is also a unit of $  \mathbb{Z}^{11} $, and it has the unique form
$$
(\pm 1, \pm 1, \pm 1, \pm 1, \pm 1, \pm 1, \pm 1, \pm 1, \pm 1, \pm 1, \pm 1)^T $$
It can be checked with the help of the characters, that the unit group of $ \mathcal{R}(G) =\{\pm 1, \pm x_1\}$, is isomorphic to the Klein four-group.
\end{proof}

\begin{remark}
Based on the idea introduced in \ref{group}, we understand that for relatively small values of \( n \), the unit group of \( \mathcal{R}(S_n) \) is isomorphic to the Klein four-group (also known as the Vierergruppe). However, the generalization of this result to all positive integers \( n \) necessitates alternative approaches for verification.
\end{remark}

Let us now discuss the representation ring $\mathcal{R}(S_6)$ as a structure of a semi-simple commutative ring. For any finite group \( G \), let its conjugate class be \( \mathcal{C}_i \) (\( 1 \leq i \leq s \)), and we can define the characteristic function \( e_i \) (\( 1 \leq i \leq s \)) satisfying
\[
e_i(g)=\left\{
\begin{array}{ccc}
1,\quad g\in\mathcal{C}_i, \\
0,\quad g\notin \mathcal{C}_i.
\end{array}
\right.
\]

They are then spanded linearly to the elements in \( \mathcal{R}_{ \mathbb{F}}(G) \), where \( \mathbb{F} \) is the splitting field of \( G \). From the definition of \( e_i \) we know that \( \chi_i = \sum_{j=1}^{s} \chi_{ij}e_j \) and \( \chi_ie_j = \chi_{ij}e_j \). In other words, they are the common eigenvectors under the regular representation of the representation algebra \( \mathbb{F}(G) \).

By considering the second orthogonality relation of the characters
$\frac{|\mathcal{C}_i|}{|G|}\sum_{j=1}^s\overline{\chi_{ji}}\chi_{jk}=\delta_{ki}$, we have
\begin{align}\label{gongshi}
e_i=\frac{|\mathcal{C}_i|}{|G|}\sum_{j,k=1}^s\overline{\chi_{ji}}\chi_{jk}e_k=\frac{|\mathcal{C}_i|}{|G|}\sum_{j=1}^s\overline{\chi_{ji}}\chi_{j},\quad\forall 1\leq i\leq s.
\end{align}

The characteristic function \( e_i \) (\( 1 \leq i \leq s \)) gives all the central primitive idempotents of \( \mathcal{R}_{\mathbb{F}}(G) \).


\begin{proposition}
The representation ring $\mathcal{R}(S_6)$ can be decomposed into the direct sum of eleven one-dimensional representations, which are spanned by the following:
\begin{align*}
&\alpha_1=1+x_1+5x_2+5x_3+9x_4+9x_5+10x_6+10x_7+5x_8+5x_9+16x_{10}\\
&\alpha_2=1-x_1+3x_2-3x_3+3x_4-3x_5+2x_6-2x_7+x_8-x_9\\
&\alpha_3=1-x_1-x_6+x_7+x_8-x_9\\
&\alpha_4=1-x_1-x_2+x_3+3x_4-3x_5-2x_6+2x_7-3x_8+3x_9\\
&\alpha_5=1-x_1+x_2-x_3-x_4+x_5-x_8+x_9\\
&\alpha_6=1-x_1-x_2+x_3+x_6-x_7\\
&\alpha_7=1+x_1+x_2+x_3+x_4+x_5-2x_6-2x_7+x_8+x_9\\
&\alpha_8=1+x_1+2x_2+2x_3+x_6+x_7-x_8-x_9-2x_{10}\\
&\alpha_9=1+x_1-x_4-x_5+x_{10}\\
&\alpha_{10}=1+x_1-x_2-x_3+x_4+x_5-x_8-x_9\\
&\alpha_{11}=1+x_1-x_2-x_3+x_6+x_7+2x_8+2x_9-2x_{10}
\end{align*}
Thus the representation ring $\mathcal{R}(S_6)$ has the following central primitive idempotents:
$$
e_1=\frac{1}{720}\alpha_1,\quad e_2=\frac{1}{48}\alpha_2,\quad e_3=\frac{1}{6}\alpha_3,\quad
e_4=\frac{1}{48}\alpha_4,\quad e_5=\frac{1}{8}\alpha_5,
$$
$$
e_6=\frac{1}{6}\alpha_6,\quad e_7=\frac{1}{16}\alpha_7,\quad 
e_8=\frac{1}{18}\alpha_8,\quad e_9=\frac{1}{5}\alpha_9,\quad e_{10}=\frac{1}{8}\alpha_{10},\quad e_{11}=\frac{1}{18}\alpha_{11}.
$$
\end{proposition}
For the  representation ring  $\mathcal{R}(S_6)$, let, $0\leq i\leq 10$, and $x_0=1$,
$$
x_i\cdot(1,x_1,x_2,x_3,x_4,x_5,x_6,x_7,x_8,x_9,x_{10})=(1,x_1,x_2,x_3,x_4,x_5,x_6,x_7,x_8,x_9,x_{10})X_{i+1},
$$
Then we have $X_1$ as the identity matrix $I_{11}$, and
{\small \begin{align*}
X_2=\left(\begin{matrix}
0 & 1 & 0 & 0 & 0 & 0 & 0 & 0 & 0 & 0\quad 0 \\
1 & 0 & 0 & 0 & 0 & 0 & 0 & 0 & 0 & 0\quad 0 \\
0 & 0 & 0 & 1 & 0 & 0 & 0 & 0 & 0 & 0\quad 0 \\
0 & 0 & 1 & 0 & 0 & 0 & 0 & 0 & 0 & 0\quad 0 \\
0 & 0 & 0 & 0 & 0 & 1 & 0 & 0 & 0 & 0\quad 0 \\
0 & 0 & 0 & 0 & 1 & 0 & 0 & 0 & 0 & 0\quad 0 \\
0 & 0 & 0 & 0 & 0 & 0 & 0 & 1 & 0 & 0\quad 0 \\
0 & 0 & 0 & 0 & 0 & 0 & 1 & 0 & 0 & 0\quad 0 \\
0 & 0 & 0 & 0 & 0 & 0 & 0 & 0 & 0 & 1\quad 0 \\
0 & 0 & 0 & 0 & 0 & 0 & 0 & 0 & 1 & 0\quad 0 \\
0 & 0 & 0 & 0 & 0 & 0 & 0 & 0 & 0 & 0\quad 1
\end{matrix}\right),\quad
X_3=\left(\begin{matrix}
0 & 0 & 1 & 0 & 0 & 0 & 0 & 0 & 0 & 0\quad 0 \\
0 & 0 & 0 & 1 & 0 & 0 & 0 & 0 & 0 & 0\quad 0 \\
1 & 0 & 1 & 0 & 1 & 0 & 1 & 0 & 0 & 0\quad 0 \\
0 & 1 & 0 & 1 & 0 & 1 & 0 & 1 & 0 & 0\quad 0 \\
0 & 0 & 1 & 0 & 1 & 0 & 1 & 0 & 1 & 0 \quad 1 \\
0 & 0 & 0 & 1 & 0 & 1 & 0 & 1 & 0 & 1 \quad 1 \\
0 & 0 & 1 & 0 & 1 & 0 & 1 & 1 & 0 & 0 \quad 1 \\
0 & 0 & 0 & 1 & 0 & 1 & 1 & 1 & 0 & 0 \quad 1 \\
0 & 0 & 0 & 0 & 1 & 0 & 0 & 0 & 0 & 0 \quad 1 \\
0 & 0 & 0 & 0 & 0 & 1 & 0 & 0 & 0 & 0 \quad 1 \\
0 & 0 & 0 & 0 & 1 & 1 & 1 & 1 & 1 & 1 \quad 2
\end{matrix}\right),
\end{align*}
\begin{align*}
X_4=\left(\begin{matrix}
0 & 0 & 0 & 1 & 0 & 0 & 0 & 0 & 0 & 0\quad 0 \\
0 & 0 & 1 & 0 & 0 & 0 & 0 & 0 & 0 & 0 \quad 0 \\
0 & 1 & 0 & 1 & 0 & 1 & 0 & 1 & 0 & 0 \quad  0 \\
1 & 0 & 1 & 0 & 1 & 0 & 1 & 0 & 0 & 0 \quad  0 \\
0 & 0 & 0 & 1 & 0 & 1 & 0 & 1 & 0 & 1 \quad  1 \\
0 & 0 & 1 & 0 & 1 & 0 & 1 & 0 & 1 & 0 \quad  1 \\
0 & 0 & 0 & 1 & 0 & 1 & 1 & 1 & 0 & 0 \quad  1 \\
0 & 0 & 1 & 0 & 1 & 0 & 1 & 1 & 0 & 0 \quad  1 \\
0 & 0 & 0 & 0 & 0 & 1 & 0 & 0 & 0 & 0 \quad  1 \\
0 & 0 & 0 & 0 & 1 & 0 & 0 & 0 & 0 & 0 \quad  1 \\
0 & 0 & 0 & 0 & 1 & 1 & 1 & 1 & 1 & 1 \quad  2
\end{matrix}\right),\quad
X_5=\left(\begin{matrix}
0 & 0 & 0 & 0 & 1 & 0 & 0 & 0 & 0 & 0 \quad 0 \\
0 & 0 & 0 & 0 & 0 & 1 & 0 & 0 & 0 & 0 \quad 0 \\
0 & 0 & 1 & 0 & 1 & 0 & 1 & 0 & 1 & 0 \quad 1 \\
0 & 0 & 0 & 1 & 0 & 1 & 0 & 1 & 0 & 1 \quad 1 \\
1 & 0 & 1 & 0 & 2 & 0 & 1 & 1 & 0 & 1 \quad 2 \\
0 & 1 & 0 & 1 & 0 & 2 & 1 & 1 & 1 & 0 \quad 2 \\
0 & 0 & 1 & 0 & 1 & 1 & 2 & 1 & 1 & 0 \quad 2 \\
0 & 0 & 0 & 1 & 1 & 1 & 1 & 2 & 0 & 1 \quad 2 \\
0 & 0 & 1 & 0 & 0 & 1 & 1 & 0 & 1 & 0 \quad 1 \\
0 & 0 & 0 & 1 & 1 & 0 & 0 & 1 & 0 & 1 \quad 1 \\
0 & 0 & 1 & 1 & 2 & 2 & 2 & 2 & 1 & 1 \quad 3
\end{matrix}\right),
\end{align*}
\begin{align*}
X_6=\left(\begin{matrix}
0 & 0 & 0 & 0 & 0 & 1 & 0 & 0 & 0 & 0 \quad 0 \\
0 & 0 & 0 & 0 & 1 & 0 & 0 & 0 & 0 & 0 \quad 0 \\
0 & 0 & 0 & 1 & 0 & 1 & 0 & 1 & 0 & 1 \quad 1 \\
0 & 0 & 1 & 0 & 1 & 0 & 1 & 0 & 1 & 0 \quad 1 \\
0 & 1 & 0 & 1 & 0 & 2 & 1 & 1 & 1 & 0 \quad 2 \\
1 & 0 & 1 & 0 & 2 & 0 & 1 & 1 & 0 & 1 \quad 2 \\
0 & 0 & 0 & 1 & 1 & 1 & 1 & 2 & 0 & 1 \quad 2 \\
0 & 0 & 1 & 0 & 1 & 1 & 2 & 1 & 1 & 0 \quad 2 \\
0 & 0 & 0 & 1 & 1 & 0 & 0 & 1 & 0 & 1 \quad 1 \\
0 & 0 & 1 & 0 & 0 & 1 & 1 & 0 & 1 & 0 \quad 1 \\
0 & 0 & 1 & 1 & 2 & 2 & 2 & 2 & 1 & 1 \quad 3
\end{matrix}\right),\quad
X_7=\left(\begin{matrix}
0 & 0 & 0 & 0 & 0 & 0 & 1 & 0 & 0 & 0 \quad 0 \\
0 & 0 & 0 & 0 & 0 & 0 & 0 & 1 & 0 & 0 \quad 0 \\
0 & 0 & 1 & 0 & 1 & 0 & 1 & 1 & 0 & 0 \quad 1 \\
0 & 0 & 0 & 1 & 0 & 1 & 1 & 1 & 0 & 0 \quad 1 \\
0 & 0 & 1 & 0 & 1 & 1 & 2 & 1 & 1 & 0 \quad 2 \\
0 & 0 & 0 & 1 & 1 & 1 & 1 & 2 & 0 & 1 \quad 2 \\
1 & 0 & 1 & 1 & 2 & 1 & 1 & 1 & 1 & 1 \quad 2 \\
0 & 1 & 1 & 1 & 1 & 2 & 1 & 1 & 1 & 1 \quad 2 \\
0 & 0 & 0 & 0 & 1 & 0 & 1 & 1 & 0 & 1 \quad 1 \\
0 & 0 & 0 & 0 & 0 & 1 & 1 & 1 & 1 & 0 \quad 1 \\
0 & 0 & 1 & 1 & 2 & 2 & 2 & 2 & 1 & 1 \quad 4
\end{matrix}\right),
\end{align*}
\begin{align*}
X_8=\left(\begin{matrix}
0 & 0 & 0 & 0 & 0 & 0 & 0 & 1 & 0 & 0 \quad 0 \\
0 & 0 & 0 & 0 & 0 & 0 & 1 & 0 & 0 & 0 \quad 0 \\
0 & 0 & 0 & 1 & 0 & 1 & 1 & 1 & 0 & 0 \quad 1 \\
0 & 0 & 1 & 0 & 1 & 0 & 1 & 1 & 0 & 0 \quad 1 \\
0 & 0 & 0 & 1 & 1 & 1 & 1 & 2 & 0 & 1 \quad 2 \\
0 & 0 & 1 & 0 & 1 & 1 & 2 & 1 & 1 & 0 \quad 2 \\
0 & 1 & 1 & 1 & 1 & 2 & 1 & 1 & 1 & 1 \quad 2 \\
1 & 0 & 1 & 1 & 2 & 1 & 1 & 1 & 1 & 1 \quad 2 \\
0 & 0 & 0 & 0 & 0 & 1 & 1 & 1 & 1 & 0 \quad 1 \\
0 & 0 & 0 & 0 & 1 & 0 & 1 & 1 & 0 & 1 \quad 1 \\
0 & 0 & 1 & 1 & 2 & 2 & 2 & 2 & 1 & 1 \quad 4
\end{matrix}\right),\quad
X_9=\left(\begin{matrix}
0 & 0 & 0 & 0 & 0 & 0 & 0 & 0 & 1 & 0 \quad 0 \\
0 & 0 & 0 & 0 & 0 & 0 & 0 & 0 & 0 & 1 \quad 0 \\
0 & 0 & 0 & 0 & 1 & 0 & 0 & 0 & 0 & 0 \quad 1 \\
0 & 0 & 0 & 0 & 0 & 1 & 0 & 0 & 0 & 0 \quad 1 \\
0 & 0 & 1 & 0 & 0 & 1 & 1 & 0 & 1 & 0 \quad 1 \\
0 & 0 & 0 & 1 & 1 & 0 & 0 & 1 & 0 & 1 \quad 1 \\
0 & 0 & 0 & 0 & 1 & 0 & 1 & 1 & 0 & 1 \quad 1 \\
0 & 0 & 0 & 0 & 0 & 1 & 1 & 1 & 1 & 0 \quad 1 \\
1 & 0 & 0 & 0 & 1 & 0 & 0 & 1 & 0 & 1 \quad 0 \\
0 & 1 & 0 & 0 & 0 & 1 & 1 & 0 & 1 & 0 \quad 0 \\
0 & 0 & 1 & 1 & 1 & 1 & 1 & 1 & 0 & 0 \quad 2
\end{matrix}\right),
\end{align*}
\begin{align*}
X_{10}=\left(\begin{matrix}
0 & 0 & 0 & 0 & 0 & 0 & 0 & 0 & 0 & 1 \quad 0 \\
0 & 0 & 0 & 0 & 0 & 0 & 0 & 0 & 1 & 0 \quad 0 \\
0 & 0 & 0 & 0 & 0 & 1 & 0 & 0 & 0 & 0 \quad 1 \\
0 & 0 & 0 & 0 & 1 & 0 & 0 & 0 & 0 & 0 \quad 1 \\
0 & 0 & 0 & 1 & 1 & 0 & 0 & 1 & 0 & 1 \quad 1 \\
0 & 0 & 1 & 0 & 0 & 1 & 1 & 0 & 1 & 0 \quad 1 \\
0 & 0 & 0 & 0 & 0 & 1 & 1 & 1 & 1 & 0 \quad 1 \\
0 & 0 & 0 & 0 & 1 & 0 & 1 & 1 & 0 & 1 \quad 1 \\
0 & 1 & 0 & 0 & 0 & 1 & 1 & 0 & 1 & 0 \quad 0 \\
1 & 0 & 0 & 0 & 1 & 0 & 0 & 1 & 0 & 1 \quad 0 \\
0 & 0 & 1 & 1 & 1 & 1 & 1 & 1 & 0 & 0 \quad 2
\end{matrix}\right),\quad
X_{11}=\left(\begin{matrix}
0 & 0 & 0 & 0 & 0 & 0 & 0 & 0 & 0 & 0 \quad 1 \\
0 & 0 & 0 & 0 & 0 & 0 & 0 & 0 & 0 & 0 \quad 1 \\
0 & 0 & 0 & 0 & 1 & 1 & 1 & 1 & 1 & 1 \quad 2 \\
0 & 0 & 0 & 0 & 1 & 1 & 1 & 1 & 1 & 1 \quad 2 \\
0 & 0 & 1 & 1 & 2 & 2 & 2 & 2 & 1 & 1 \quad 3 \\
0 & 0 & 1 & 1 & 2 & 2 & 2 & 2 & 1 & 1 \quad 3 \\
0 & 0 & 1 & 1 & 2 & 2 & 2 & 2 & 1 & 1 \quad 4 \\
0 & 0 & 1 & 1 & 2 & 2 & 2 & 2 & 1 & 1 \quad 4 \\
0 & 0 & 1 & 1 & 1 & 1 & 1 & 1 & 0 & 0 \quad 2 \\
0 & 0 & 1 & 1 & 1 & 1 & 1 & 1 & 0 & 0 \quad 2 \\
1 & 1 & 2 & 2 & 3 & 3 & 4 & 4 & 2 & 2 \quad 5
\end{matrix}\right).
\end{align*}}


Let us proceed with computing the Casimir number and determinant of $\mathcal{R}(S_6)$ for example
\begin{align*}
X&=X_1^2+X_2^2+X_3^2+X_4^2+X_5^2+X_6^2+X_7^2+X_8^2+X_9^2+X_{10}^2+X_{11}^2
\end{align*}
\begin{center}
\renewcommand{\arraystretch}{1.5} 
 $=\left(
\begin{tabular}{*{11}{c}}
11 & 1 & 8 & 4 & 15 & 5 & 10 & 10 & 4 & 8 & 13 \\
1 & 11 & 4 & 8 & 5 & 15 & 10 & 10 & 8 & 4 & 13 \\
8 & 4 & 44 & 20 & 50 & 40 & 56 & 42 & 28 & 18 & 78 \\
4 & 8 & 20 & 44 & 40 & 50 & 42 & 56 & 18 & 28 & 78 \\
15 & 5 & 50 & 40 & 103 & 65 & 88 & 88 & 40 & 50 & 143 \\
5 & 15 & 40 & 50 & 65 & 103 & 88 & 88 & 50 & 40 & 143 \\
10 & 10 & 56 & 42 & 88 & 88 & 116 & 96 & 56 & 42 & 156 \\
10 & 10 & 42 & 56 & 88 & 88 & 96 & 116 & 42 & 56 & 156 \\
4 & 8 & 28 & 18 & 40 & 50 & 56 & 42 & 44 & 20 & 78 \\
8 & 4 & 18 & 28 & 50 & 40 & 42 & 56 & 20 & 44 & 78 \\
13 & 13 & 78 & 78 & 143 & 143 & 156 & 156 & 78 & 78 & 265 \\
\end{tabular}
\right)$,
\end{center}

On the rational number field $\mathbb Q$ we have the inverse matrix $X^{-1}$.
\begin{center}
\renewcommand{\arraystretch}{1.5} 
 $X^{-1}=\left(
\begin{tabular}{*{11}{c}}
$\frac{5951}{43200}$ & $\frac{-137}{21600}$ & $\frac{-43}{2160}$ & $\frac{23}{8640}$ & $\frac{-241}{7200}$ & $\frac{-107}{14400}$ & $\frac{-7}{4320}$ & $\frac{-7}{4320}$ & $\frac{23}{8640}$ & $\frac{-43}{2160}$ & $\frac{299}{10800}$ \\
$\frac{-137}{21600}$ & $\frac{5951}{43200}$ & $\frac{23}{8640}$ & $\frac{-43}{2160}$ & $\frac{-107}{14400}$ & $\frac{-241}{7200}$ & $\frac{-7}{4320}$ & $\frac{-7}{4320}$ & $\frac{-43}{2160}$ & $\frac{23}{8640}$ & $\frac{299}{10800}$ \\
$\frac{-43}{2160}$ & $\frac{23}{8640}$ & $\frac{143}{1728}$ & $\frac{-11}{864}$ & $\frac{-71}{2880}$ & $\frac{1}{720}$ & $\frac{-25}{864}$ & $\frac{17}{864}$ & $\frac{-5}{864}$ & $\frac{35}{1728}$ & $\frac{-13}{2160}$ \\
$\frac{23}{8640}$ & $\frac{-43}{2160}$ & $\frac{-11}{864}$ & $\frac{143}{1728}$ & $\frac{1}{720}$ & $\frac{-71}{2880}$ & $\frac{17}{864}$ & $\frac{-25}{864}$ & $\frac{35}{1728}$ & $\frac{-5}{864}$ & $\frac{-13}{2160}$ \\
$\frac{-241}{7200}$ & $\frac{-107}{14400}$ & $\frac{-71}{2880}$ & $\frac{1}{720}$ & $\frac{133}{1600}$ & $\frac{29}{800}$ & $\frac{-11}{1440}$ & $\frac{-11}{1440}$ & $\frac{1}{720}$ & $\frac{-71}{2880}$ & $\frac{-143}{3600}$ \\
$\frac{-107}{14400}$ & $\frac{-241}{7200}$ & $\frac{1}{720}$ & $\frac{-71}{2880}$ & $\frac{29}{800}$ & $\frac{133}{1600}$ & $\frac{-11}{1440}$ & $\frac{-11}{1440}$ & $\frac{-71}{2880}$ & $\frac{1}{720}$ & $\frac{-143}{3600}$ \\
$\frac{-7}{4320}$ & $\frac{-7}{4320}$ & $\frac{-25}{864}$ & $\frac{17}{864}$ & $\frac{-11}{1440}$ & $\frac{-11}{1440}$ & $\frac{35}{432}$ & $\frac{-1}{27}$ & $\frac{-25}{864}$ & $\frac{17}{864}$ & $\frac{-13}{1080}$ \\
$\frac{-7}{4320}$ & $\frac{-7}{4320}$ & $\frac{17}{864}$ & $\frac{-25}{864}$ & $\frac{-11}{1440}$ & $\frac{-11}{1440}$ & $\frac{-1}{27}$ & $\frac{35}{432}$ & $\frac{17}{864}$ & $\frac{-25}{864}$ & $\frac{-13}{1080}$ \\
$\frac{23}{8640}$ & $\frac{-43}{2160}$ & $\frac{-5}{864}$ & $\frac{35}{1728}$ & $\frac{1}{720}$ & $\frac{-71}{2880}$ & $\frac{-25}{864}$ & $\frac{17}{864}$ & $\frac{143}{1728}$ & $\frac{-11}{864}$ & $\frac{-13}{2160}$ \\
$\frac{-43}{2160}$ & $\frac{23}{8640}$ & $\frac{35}{1728}$ & $\frac{-5}{864}$ & $\frac{-71}{2880}$ & $\frac{1}{720}$ & $\frac{17}{864}$ & $\frac{-25}{864}$ & $\frac{-11}{864}$ & $\frac{143}{1728}$ & $\frac{-13}{2160}$ \\
$\frac{299}{10800}$ & $\frac{299}{10800}$ & $\frac{-13}{2160}$ & $\frac{-13}{2160}$ & $\frac{-143}{3600}$ & $\frac{-143}{3600}$ & $\frac{-13}{1080}$ & $\frac{-13}{1080}$ & $\frac{-13}{2160}$ & $\frac{-13}{2160}$ & $\frac{44}{675}$ \\
\end{tabular}
\right)$,
\end{center}

So the Casimir number of $\mathcal{R}(S_6)$ is $43200=2^6\cdot 3^3\cdot 5^2$. \par


On the other hand, the determinant of $\mathcal{R}(S_6)$ is $\det X= 2^{14}\cdot 3^5\cdot 5^2\cdot 7\cdot 11\cdot 13$. The two non-negative integers share the same prime factors (see also \cite{WLL} for the result). Since the Casimir number and the determinant of $\mathcal{R}(S_6)$ are non-zero, we know from the results in \cite{WLZ} that the representation ring is semi-simple.

\section{Harrison centers}\label{5}
The equivalence of tensor products, outlined in Appendix I, raises a provocative question: Is it possible to study representation rings through the lens of polynomial theory? With this in mind, we propose a view that examines the interplay between Harrison centers and representation rings, potentially enriching our understanding of the algebraic framework in group representations. Throughout this section, let $d \geq 3$  and $\mathbb{K}$ be a field of characteristic 0 or $\geq 3$. As a preparation, let us review some notations of Harrison center theory.

Let $f(x_1, x_2, \ldots, x_n) \in \mathbb{K}[x_1, \ldots, x_n]$ represent a homogeneous polynomial of degree $d$. Harrison [5] introduced the concept of the center of $f$, denoted by $Z(f)$, which is defined as\[ Z(f) := \{X \in \mathbb{K}^{n \times n} \mid HX = X^T H \}. \]
Here, $H$ is the Hessian matrix of $f$ with entries $\left( \frac{\partial^2 f}{\partial x_i \partial x_j} \right)_{1 \leq i, j \leq n}$.

\begin{example}
Let us consider the polynomial $f = \sum_{i=1}^{n} x_i^d$, where each $x_i$ is raised to the power $d$. The Hessian matrix associated with $f$ is diagonal, with diagonal entries given by $d(d-1)x_i^{d-2}$. For a matrix $X = (c_{ij})_{n\times n}$, the $ij$-entry of the product $HX$ is $d(d-1)x_i^{d-2}c_{ij}$. The symmetry of $HX$ implies that $Z(f)$ is composed exclusively of diagonal matrices. As a result, the algebraic structure of $Z(f)$ is isomorphic to $\mathbb{K}^n$.
\end{example}

Since the Harrison centers and the symmetric groups have a great deal of symmetry, this inspired us to try to study their relations to the representation ring. Specifically, we consider the relationship between the generating relations of the representation ring of the symmetric group $S_n$ and the Harrison centers of the polynomial $f$ induced by $S_n$.


\begin{theorem}\label{iso}
For the symmetric group \( S_4 \), the Harrison center \( Z(f) \) of the cubic form \( f \) is generated by the same relations as the elements of its representation ring \( \mathcal{R}(S_4)\). This implies that \( \mathcal{R}(S_4) \) is isomorphic to the Harrison center \( Z(f_{S_4}) \).
\end{theorem}



\begin{proof}
Firstly, let us look at the irreducible characters of the symmetric group $S_4$:
\begin{longtable}{|c|c|c|c|c|c|}
\hline
\diagbox{$Irr_{\mathbb C}S_4$}{$\xi$}                                                                     & \begin{tabular}[c]{@{}c@{}}$g_1$\end{tabular}                            & \begin{tabular}[c]{@{}c@{}}$g_2$\end{tabular}                              & \begin{tabular}[c]{@{}c@{}}$g_3$\end{tabular}                             & \begin{tabular}[c]{@{}c@{}}$g_4$\end{tabular}                              & \begin{tabular}[c]{@{}c@{}}$g_5$\end{tabular}                              \endfirsthead
\hline
\begin{tabular}[c]{@{}c@{}}$\chi_1$\\ $\chi_2$\\ $\chi_3$\\ $\chi_4$\\ $\chi_5$\end{tabular} & \begin{tabular}[c]{@{}c@{}}1\\1\\2\\3\\3\end{tabular} & \begin{tabular}[c]{@{}c@{}}1\\-1\\0\\1\\-1\end{tabular} & \begin{tabular}[c]{@{}c@{}}1\\1\\-1\\0\\0\end{tabular} & \begin{tabular}[c]{@{}c@{}}1\\-1\\0\\-1\\1\end{tabular} & \begin{tabular}[c]{@{}c@{}}1\\1\\2\\-1\\-1\end{tabular} \\
\hline
\caption{The irreducible characters of symmetric group $S_4$}
\end{longtable}
From the table above, it is very easy to check the relations 
:
\begin{align*}
&V_1\otimes V_i\cong V_i\otimes V_1,\quad \text{where $i\in\{1,2,3,4,5\}$},\\
&V_2\otimes V_2\cong V_1, V_2\otimes V_3\cong V_3\otimes V_2\cong V_3, V_3\otimes V_3\cong V_1\oplus V_2\oplus V_3,\\
&V_2\otimes V_4\cong V_4\otimes V_2\cong V_5, V_2\otimes V_5\cong V_5\otimes V_2\cong V_4,\\
&V_3\otimes V_4\cong V_4\otimes V_3\cong V_3\otimes V_5\cong V_5\otimes V_3\cong V_4\oplus V_5,\\
&V_4\otimes V_4\cong V_5\otimes V_5\cong V_1\oplus V_3\oplus V_4\oplus V_5, V_4\otimes V_5\cong V_5\otimes V_4\cong V_2\oplus V_3\oplus V_4\oplus V_5.
\end{align*}
Its induced polynomial $f_{S_4}$ can be written in terms of the relations given above,
\begin{align*}
f_{S_4}&=x_1^3+x_3^3+x_4^3+x_5^3+3x_1x_2^2+3x_1x_3^2+3x_1x_4^2+3x_1x_5^2+3x_2x_3^2+3x_3x_4^2+3x_3x_5^2\\
&+3x_4^2x_5+3x_4x_5^2+6x_2x_4x_5+6x_3x_4x_5.
\end{align*}
The corresponding Hessian matrix $H_{f_{S_4}}$ has the following form
\begin{align*}
\begin{pmatrix}
6x_1 & 6x_2 & 6x_3 &6x_4 &6x_5 \\
6x_2 & 6x_1 & 6x_3 &6x_5 &6x_4 \\
6x_3 & 6x_3 & 6x_1+6x_2+6x_3 &6x_4+6x_5 &6x_4+6x_5 \\
6x_4 & 6x_5 & 6x_4+6x_5 &6x_1+6x_3+6x_4+6x_5 &6x_2+6x_3+6x_4+6x_5 \\
6x_5 & 6x_4 & 6x_4+6x_5 &6x_2+6x_3+6x_4+6x_5 &6x_1+6x_3+6x_4+6x_5
\end{pmatrix}.
\end{align*}
Now we will determine the matrix form $X_4$ of Harrison center $Z(f)$, which is $\{X_4\in \mathbb C^{3\times 3}|X_4^TH_{f_{S_4}}=H_{f_{S_4}} X_4\}$. Then we let
\begin{align*}
X_4=\begin{pmatrix}
a & b & c & d &e \\
f & g & h & i &j \\
k & l & m & n & o \\
p & q & r & s &t \\
u & v & w & x &y
\end{pmatrix}.
\end{align*}
By plugging $H_{f_{S_4}}$ and $X_4$ in $X^TH_f=H_f X$, we have
{\small
\begin{align*}
X_4^T\begin{pmatrix}
1 &0 & 0& 0& 0\\
0 &1 & 0& 0& 0\\
0 &0 &1 &0 &0\\
0 &0 &0 &1 &0\\
0 &0 &0 &0 &1
\end{pmatrix}&=\begin{pmatrix}
1 &0 & 0& 0& 0\\
0 &1 & 0& 0& 0\\
0 &0 &1 &0 &0\\
0 &0 &0 &1 &0\\
0 &0 &0 &0 &1
\end{pmatrix}X_4,
X_4^T\begin{pmatrix}
0& 1& 0 &0 &0\\
1& 0& 0& 0& 0\\
0& 0& 1& 0& 0\\
0& 0& 0& 0& 1\\
0& 0& 0& 1& 0
\end{pmatrix}=\begin{pmatrix}
0& 1& 0 &0 &0\\
1& 0& 0& 0& 0\\
0& 0& 1& 0& 0\\
0& 0& 0& 0& 1\\
0& 0& 0& 1& 0
\end{pmatrix}X_4,\\
X_4^T\begin{pmatrix}
0& 0& 1& 0& 0\\
0& 0& 1&0 &0\\
1& 1& 1& 0& 0\\
0& 0& 0& 1& 1\\
0& 0& 0& 1& 1
\end{pmatrix}&=\begin{pmatrix}
0& 0& 1& 0& 0\\
0& 0& 1&0 &0\\
1& 1& 1& 0& 0\\
0& 0& 0& 1& 1\\
0& 0& 0& 1& 1
\end{pmatrix}X_4,
X_4^T\begin{pmatrix}
0& 0& 0& 1& 0\\
0& 0& 0& 0& 1\\
0& 0& 0& 1& 1\\
1& 0& 1& 1& 1\\
0& 1& 1& 1& 1
\end{pmatrix}=\begin{pmatrix}
0& 0& 0& 1& 0\\
0& 0& 0& 0& 1\\
0& 0& 0& 1& 1\\
1& 0& 1& 1& 1\\
0& 1& 1& 1& 1
\end{pmatrix}X_4,\\
X_4^T\begin{pmatrix}
0& 0& 0& 0& 1\\
0& 0& 0& 1& 0\\
0& 0& 0& 1& 1\\
0& 1& 1& 1& 1\\
1& 0& 1& 1& 1
\end{pmatrix}&=\begin{pmatrix}
0& 0& 0& 0& 1\\
0& 0& 0& 1& 0\\
0& 0& 0& 1& 1\\
0& 1& 1& 1& 1\\
1& 0& 1& 1& 1
\end{pmatrix}X_4.
\end{align*}}
Finally we find the expression for 
$X_4$ with respect to the parameters $a,b,c,d,e$ as
\begin{align*}
X_4=\begin{pmatrix}
a & b & c & d & e\\
b & a & c & e & d\\
c & c &a+b+c &d+e &d+e\\
d & e &d+e &a+c+d+e &b+c+d+e\\
e & d &d+e &b+c+d+e &a+c+d+e
\end{pmatrix}.
\end{align*}
Now it is obvious that the Harrison center $Z (f)$  is isomorphic to the representation ring $\mathcal{R}(S_4)$.
\end{proof}

The following theorems come quickly if you do the same with the theorem \ref{iso}:

\begin{proposition}
For the symmetric group $S_5$, $ \mathcal{R}(S_5)$ is isomorphic to $ Z (f_{S_5})$.
\end{proposition}

\begin{proposition}
For the symmetric group $S_6$, $ \mathcal{R}(S_6)$ is isomorphic to $ Z (f_{S_6})$.
\end{proposition}

\begin{remark}
    However, our method relies heavily on the characters of the irreducible representations, so for higher rank the results require a more general approach to prove (just like the case about its unit group above).
\end{remark}


\section{Remaining problems}\label{6}
\begin{enumerate}
    
    \item We have checked cases in low dimension using the method in \ref{group}. But the computation in its proof relies mainly on the specific value of the table of complex irreducible characters, which clearly does not work for the $n$ case. So we want an alternative method to prove the following result. 
    \begin{conjecture}
        For any $n\in\mathbb{N}$, the unit group $U(\mathcal{R}(S_n))$ of the representation ring of real numbers $\mathcal{R}(S_n)$ is isomorphic to the Klein four-group.
    \end{conjecture}
    
    \item Similar to the arguments above, a way of dealing with Harrison centers that is independent of certain irreducible characters should be developed.
    
    \begin{conjecture}
   $ \mathcal{R}(S_n)$ is isomorphic to $ Z (f_{S_n})$ for any $n\in\mathbb{N}$, .
    \end{conjecture}

    \item Another untreated case is the alternating groups $A_n$, we believe they have similar results. We leave that to the interested reader.
    
    \item We are also aware that there is another picture that hasn't been studied yet, the so-called {\em (Bi)Hom-groups}, as a new merging field, its representation ring. This is also a possible area of research for the interested reader. (See some related literature in \cite{LMT,SWZZ,Y})
\end{enumerate}

\clearpage

\section{Appendix I: Equivalence relations}\label{fulua}
In this appendix, we will list in detail the direct sum decomposition results of all irreducible representations of tensor products of the symmetric group $S_6$ as follows:\par
{\small
For any $V_1\otimes V_i$, $i\in \{1,2,3,4,5,6,7,8,9,10,11\}$, we have
\begin{align*}
V_1 \otimes V_i &\cong V_i \otimes V_1 \cong V_i
\end{align*}

For any $V_2\otimes V_i$, $i\in \{2,3,4,5,6,7,8,9,10,11\}$ we have
\begin{align*}
&V_2\otimes V_i=\left\{
\begin{aligned}
    & V_{i-1},\quad\text{when $i$ is even}, \\\
    & V_{i+1},\quad\text{when $i$ is odd}, \end{aligned}
\right.\\\
&V_2\otimes V_{11}\cong V_{11}\otimes V_2\cong V_{11}.
\end{align*}

For any $V_3\otimes V_i$, $i\in \{3,4,5,6,7,8,9,10,11\}$, we have
\begin{align*}
&V_3\otimes V_3\cong V_1\oplus V_3\oplus V_5\oplus V_7,\\
&V_3\otimes V_4\cong V_4\otimes V_3\cong V_2\oplus V_4\oplus V_6\oplus V_8,\\
&V_3\otimes V_5\cong V_5\otimes V_3\cong V_3\oplus V_5\oplus V_7\oplus V_9\oplus V_{11},\\
&V_3\otimes V_6\cong V_6\otimes V_3\cong V_4\oplus V_6\oplus V_8\oplus V_{10}\oplus V_{11},\\
&V_3\otimes V_7\cong V_7\otimes V_3\cong V_3\oplus V_5\oplus V_7\oplus V_8\oplus V_{11},\\
&V_3\otimes V_8\cong V_8\otimes V_3\cong V_4\oplus V_6\oplus V_7\oplus V_8\oplus V_{11},\\
&V_3\otimes V_9\cong V_9\otimes V_3\cong V_5\oplus V_{11}, 
V_3\otimes V_{10}\cong V_{10}\otimes V_3\cong V_6\oplus V_{11},\\
&V_3\otimes V_{11}\cong V_{11}\otimes V_3\cong V_5\oplus V_6\oplus V_7\oplus V_8\oplus V_9\oplus V_{10}\oplus 2V_{11}.
\end{align*}

For any $V_4\otimes V_i$, $i\in \{4,5,6,7,8,9,10,11\}$ we have
\begin{align*}
&V_4\otimes V_4\cong V_1\oplus V_3\oplus V_5\oplus V_7,\\\
&V_4\otimes V_5\cong V_5\otimes V_4\cong V_4\oplus V_6\oplus V_8\oplus V_{10}\oplus V_{11},\\
&V_4\otimes V_6\cong V_6\otimes V_4\cong V_3\oplus V_5\oplus V_7\oplus V_9\oplus V_{11},\\
&V_4\otimes V_7\cong V_7\otimes V_4\cong V_4\oplus V_6\oplus V_7\oplus V_8\oplus V_{11},\\
&V_4\otimes V_8\cong V_8\otimes V_4\cong V_3\oplus V_5\oplus V_7\oplus V_8\oplus V_{11},\\
&V_4\otimes V_9\cong V_9\otimes V_4\cong V_6\oplus V_{11}, \\
&V_4\otimes V_{10}\cong V_{10}\otimes V_4\cong V_5\oplus V_{11},\\
&V_4\otimes V_{11}\cong V_{11}\otimes V_4\cong V_5\oplus V_6\oplus V_7\oplus V_8\oplus V_9\oplus V_{10}\oplus 2V_{11}.
\end{align*}

For any $V_5\otimes V_i$, $i\in \{5,6,7,8,9,10,11\}$, we have
\begin{align*}
&V_5\otimes V_5\cong V_1\oplus V_3\oplus 2V_5\oplus V_7\oplus V_8\oplus V_{10}\oplus 2V_{11},\\
&V_5\otimes V_6\cong V_6\otimes V_5\cong V_2\oplus V_4\oplus 2V_6\oplus V_7\oplus V_8\oplus V_9\oplus 2V_{11},\\
&V_5\otimes V_7\cong V_7\otimes V_5\cong V_3\oplus V_5\oplus V_6\oplus 2V_7\oplus V_8\oplus V_9\oplus 2V_{11},\\
&V_5\otimes V_8\cong V_8\otimes V_5\cong V_4\oplus V_5\oplus V_6\oplus V_7\oplus 2V_8\oplus V_{10}\oplus 2V_{11},\\
&V_5\otimes V_9\cong V_9\otimes V_5\cong V_3\oplus V_6\oplus V_7\oplus V_9\oplus V_{11},\\
&V_5\otimes V_{10}\cong V_{10}\otimes V_5\cong V_4\oplus V_5\oplus V_8\oplus V_{10}\oplus V_{11},\\
&V_5\otimes V_{11}\cong V_{11}\otimes V_5\cong V_3\oplus V_4\oplus 2V_5\oplus 2V_6 \oplus 2V_7 \oplus 2V_8 \oplus V_9\oplus V_{10}\oplus 3V_{11}.
\end{align*}

For any $V_6\otimes V_i$, $i\in \{6,7,8,9,10,11\}$, we have
\begin{align*}
&V_6\otimes V_6\cong V_1\oplus V_3\oplus 2V_5\oplus V_7\oplus V_8\oplus V_{10}\oplus 2V_{11},\\
&V_6\otimes V_7\cong V_7\otimes V_6\cong V_4\oplus V_5\oplus V_6\oplus V_7\oplus 2V_8\oplus V_{10}\oplus 2V_{11},\\
&V_6\otimes V_8\cong V_8\otimes V_6\cong V_3\oplus V_5\oplus V_6\oplus 2V_7\oplus V_8\oplus V_9\oplus 2V_{11},\\
&V_6\otimes V_9\cong V_9\otimes V_6\cong V_4\oplus V_5\oplus V_8\oplus V_{10}\oplus V_{11},\\
&V_6\otimes V_{10}\cong V_{10}\otimes V_6\cong V_3\oplus V_6\oplus V_7\oplus V_9\oplus V_{11},\\
&V_6\otimes V_{11}\cong V_{11}\otimes V_6\cong V_3\oplus V_4\oplus 2V_5\oplus 2V_6 \oplus 2V_7 \oplus 2V_8 \oplus V_9\oplus V_{10}\oplus 3V_{11}.
\end{align*}

For any $V_7\otimes V_i$, $i\in \{7,8,9,10,11\}$ we have
\begin{align*}
&V_7\otimes V_7\cong V_1\oplus V_3\oplus V_4\oplus 2V_5\oplus V_6\oplus V_7\oplus V_8\oplus V_9\oplus V_{10}\oplus 2V_{11},\\
&V_7\otimes V_8\cong V_8\otimes V_7\cong V_2\oplus V_3\oplus V_4\oplus V_5\oplus 2V_6\oplus V_7\oplus V_8\oplus V_9\oplus V_{10}\oplus 2V_{11}, \\
&V_7\otimes V_9\cong V_9\otimes V_7\cong V_5\oplus V_7\oplus V_8\oplus V_{10}\oplus V_{11},\\\
&V_7\otimes V_{10}\cong V_{10}\otimes V_7\cong V_6\oplus V_7\oplus V_8\oplus V_9\oplus V_{11},\\
&V_7\otimes V_{11}\cong V_{11}\otimes V_7\cong V_3\oplus V_4\oplus 2V_5\oplus 2V_6 \oplus 2V_7 \oplus 2V_8 \oplus V_9\oplus V_{10}\oplus 4V_{11}.
\end{align*}

For any $V_8\otimes V_i$, $i\in \{8,9,10,11\}$ we have
\begin{align*}
&V_8\otimes V_8\cong V_1\oplus V_3\oplus V_4\oplus 2V_5\oplus V_6\oplus V_7\oplus V_8\oplus V_9\oplus V_{10}\oplus 2V_{11},\\\
&V_8\otimes V_9\cong V_9\otimes V_8\cong V_6\oplus V_7\oplus V_8\oplus V_9\oplus V_{11},\\
&V_8\otimes V_{10}\cong V_{10}\otimes V_8\cong V_5\oplus V_7\oplus V_8\oplus V_{10}\oplus V_{11},\\
&V_8\otimes V_{11}\cong V_{11}\otimes V_8\cong V_3\oplus V_4\oplus 2V_5\oplus 2V_6 \oplus 2V_7 \oplus 2V_8 \oplus V_9\oplus V_{10}\oplus 4V_{11}.
\end{align*}\par
For any $V_9\otimes V_i$, $i\in \{9,10,11\}$, we have
\begin{align*}
&V_9\otimes V_9\cong V_1\oplus V_5\oplus V_8\oplus V_{10},\\
&V_9\otimes V_{10}\cong V_{10}\otimes V_9\cong V_2\oplus V_6\oplus V_7\oplus V_9,\\
&V_9\otimes V_{11}\cong V_{11}\otimes V_9\cong V_3\oplus V_4\oplus V_5\oplus V_6\oplus V_7\oplus V_8\oplus 2V_{11}.
\end{align*}

For any $V_{10}\otimes V_i$,  $i\in \{10,11\}$ we have
\begin{align*}
&V_{10}\otimes V_{10}\cong V_1\oplus V_5\oplus V_8\oplus V_{10},\\
&V_{10}\otimes V_{11}\cong V_{11}\otimes V_{10}\cong V_3\oplus V_4\oplus V_5\oplus V_6\oplus V_7\oplus V_8\oplus 2V_{11}.
\end{align*}

For any $V_{11}\otimes V_i$, $i\in \{11\}$ we have
\begin{align*}
V_{11}\otimes V_{11}\cong V_1\oplus V_2\oplus 2V_3\oplus 2V_4\oplus 3V_5\oplus 3V_6\oplus 4V_7\oplus 4V_8\oplus 2V_9\oplus 2V_{10}\oplus 5V_{ 11}.
\end{align*}
}

\clearpage
\section{Appendix II: Generating relations}\label{fulub}
Here we will list the generating relations in Thm.\ref{shengcheng}.
{\small
\begin{align*}
&y_1=x_1^2-1,  \quad 
y_2=x_2x_1-x_3, \quad
y_3=x_3x_1-x_2,\\
&y_4=x_4x_1-x_5, \quad
y_5=x_5x_1-x_4, \quad
y_6=x_6x_1-x_7,\\
&y_7=x_7x_1-x_6, \quad
y_8=x_8x_6, \quad
y_9=x_9x_1-x_8,\\
&y_{10}=x_{10}x_1-x_{10}, \quad
y_{11}=x_{10}x_1-x_9, \quad y_{12}=x_3x_2-x_7-x_5-x_3-x_1,\\
&y_{13}=x_4x_2-x_{10}-x_8-x_6-x_4-x_2, \quad
y_{14}=x_5x_2-x_{10}-x_9-x_7-x_5-x_3,\\
&y_{15}=x_6x_2-x_{10}-x_7-x_6-x_4-x_2, \quad
y_{16}=x_7x_2-x_{10}-x_7-x_6-x_5-x_3,\\
&y_{17}=x_8x_2-x_{10}-x_4, \quad
y_{18}=x_9x_2-x_{10}-x_5,\\\
&y_{19}=x_{10}x_2-2x_{10}-x-9-x_8-x_7-x_7-x_6-x_5-x_4, \\
&y_{20}=x_3^2-x_6-x_4-x_2-1,\\
&y_{21}=x_4x_3-x_{10}-x_9-x_7-x_5-x_3, \quad
y_{22}=x_5x_3-x_{10}-x_8-x_6-x_4-x_2,\\
&y_{23}=x_6x_3-x_{10}-x_7-x_6-x_5-x_3, \quad
y_{24}=x_7x_3-x_{10}-x_7-x_6-x_4-x_2,\\
&y_{25}=x_8x_3-x_{10}-x_5, \quad
y_{26}=x_9x_3-x_{10}-x_4,\\
&y_{27}=x_{10}x_3-2x_{10}-x_9-x_8-x_7-x_6-x_5-x_4,\\
&y_{28}=x_4^2-2x_{10}-x_9-x_7-x_6-2x_4-x_2-1\\
&y_{29}=x_5x_4-2x_{10}-x_8-x_7-x_6-2x_5-x_3-x_1,\\
&y_{30}=x_6x_4-2x_{10}-x_8-x_7-2x_6-x_5-x_4-x_2\\
&y_{31}=x_7x_4-2x_{10}-x_9-2x_7-x_6-x_5-x_4-x_3,\\
&y_{32}=x_8x_4-x_{10}-x_8-x_6-x_5-x_2,\\
&y_{33}=x_9x_4-x_{10}-x_9-x_7-x_4-x_3,\\
&y_{34}=x_{10}x_4-3x_{10}-x_9-x_8-2x_7-2x_6-2x_5-2x_4-x_3-x_2,\\
&y_{35}=x_{5}^2-2x_{10}-x_9-x_7-x_6-2x_4-2x_4-x_2-1,\\
&y_{36}=x_6x_{5}-2x_{10}-x_9-2x_7-x_6-x_5-x_4-x_3,\\
&y_{37}=x_7x_{5}-2x_{10}-x_8-x_7-2x_6-x_5-x_4-x_2, \\
&y_{38}=x_8x_{5}-x_{10}-x_9-x_7-x_4-x_3,\\
&y_{39}=x_9x_{5}-x_{10}-x_8-x_6-x_5-x_2,\\
&y_{40}=x_{10}x_{5}-3x_{10}-x_9-x_8-2x_7-2x_6-2x_5-2x_4-x_3-x_2,\\
&y_{41}=x_{6}^2-2x_{10}-x_9-x_8-x_7-x_6-x_5-2x_4-x_3-x_2-1,\\
&y_{42}=x_7x_6-2x_{10}-x_9-x_8-x_7-x_6-2x_5-x_4-x_3-x_2-x_1,\\
&y_{43}=x_8x_6-x_{10}-x_9-x_7-x_6-x_4, \quad
y_{44}=x_9x_6-x_{10}-x_8-x_7-x_6-x_5,\\
&y_{45}=x_{10}x_6-4x_{10}-x_9-x_8-2x_7-2x_6-2x_5-2x_4-x_3-x_2,\\
&y_{46}=x_7^2-2x_{10}-x_9-x_8-x_7-x_6-x_5-2x_4-x_3-x_2-1,\\
&y_{47}=x_8x_7-x_{10}-x_8-x_7-x_6-x_5, \quad
y_{48}=x_9x_7-x_{10}-x_9-x_7-x_6-x_4,\\\
&y_{49}=x_{10}x_7-4x_{10}-x_9-x_8-2x_7-2x_6-2x_5-2x_4-x_3-x_2, \\
&y_{50}=x_8^2-x_9-x_7-x_4-1,\\
&y_{51}=x_9x_8-x_8-x_6-x_5-x_1, \quad
y_{52}=x_{10}x_8-2x_{10}-x_7-x_6-x_5-x_4-x_3-x_2,\\
&y_{53}=x_9^2-x_9-x_7-x_4-1, \quad
y_{54}=x_{10}x_9-2x_{10}-x_7-x_6-x_5-x_4-x_3-x_2,\\
&y_{55}=x_{10}^2-5x_{10}-2x_9-2x_8-4x_7-4x_6-3x_5-3x_4-2x_3-2x_2-x_1-1.
\end{align*}
}

\end{document}